\RequirePackage{fix-cm}
\documentclass[smallextended]{svjour3}       
\smartqed  
\usepackage{graphicx}
\usepackage{amsmath}
\usepackage{amsfonts}
\usepackage{bm}
\usepackage{prettyref}
\usepackage[
  unicode=true,
  bookmarks=true,
  bookmarksnumbered=false,
  bookmarksopen=false,
  breaklinks=true,
  backref=false,
  colorlinks=false    
]{hyperref}
\usepackage{orcidlink}
\usepackage{doi}

\hypersetup{
  pdfborder={0 0 0}
}

\newrefformat{fig}{Figure~\ref{#1}}
\newrefformat{sec}{Section~\ref{#1}}
\newrefformat{sub}{Section~\ref{#1}}
\newrefformat{thm}{Theorem~\ref{#1}}
\newrefformat{rem}{Remark~\ref{#1}}
\newrefformat{lm}{Lemma~\ref{#1}}
\newrefformat{cor}{Corollary~\ref{#1}}

\newcommand{\etol}{e_{\mathrm{tol}}}
\newcommand{\kmax}{k_{\mathrm{max}}}
\newcommand{\tolnewton}{tol_{\mathrm{newton}}}
\newcommand{\tolref}{tol_{\mathrm{ref}}}
\newcommand{\dtref}{\Delta t_{\mathrm{ref}}}

\begin{document}

\title{On the analysis of spectral deferred corrections for differential-algebraic equations of index one
}

\titlerunning{On the analysis of SDC for DAEs of index one}        

\author{Matthias Bolten\orcidlink{0000-0002-8682-7652}         \and
        Lisa Wimmer\orcidlink{0000-0001-8829-0978} 
}


\institute{Matthias Bolten \and Lisa Wimmer\at
              Fakult\"{a}t f\"{u}r Mathematik und Naturwissenschaften, Bergische Universit\"{a}t Wuppertal, Gaußstraße 20, 40297 Wuppertal, Germany \\
              \email{bolten@math.uni-wuppertal.de} \\
              \email{wimmer@math.uni-wuppertal.de}
}

\date{Received: date / Accepted: date}

\maketitle

\begin{abstract}
{In this paper, we present a new spectral deferred corrections (SDC) method to solve semi-explicit differential-algebraic equations (DAEs) with the ability to be parallelized. The new scheme restricts numerical integration to differential equations. In Y. Xia et al. (2007), it was shown that each correction elevates the order of the solution by one. We show that this carries over to the new SDC scheme. The derivation of the method combines the approach of SDC and the idea to enforce the algebraic constraints without numerical integration as shown in the $\varepsilon$-embedding method by E. Hairer and G. Wanner (1996). Keeping the algebraic equations as an implicit condition of the system allows an efficient solve of semi-explicit DAEs with high-accuracy. The proposed scheme is compared with other DAE methods. We demonstrate that the proposed SDC scheme is competitive with Runge-Kutta methods for DAEs in terms of accuracy and its parallelized versions are very efficient compared to their associated sequential SDC variants.}
\end{abstract}

\section{Introduction} \label{sec:intro}

Consider the general differential-algebraic equations (DAEs) formulated as implicit differential equations (IDEs)
\begin{equation}
    \bm{0} = \bm{F}(t, \bm{u}(t), \bm{u}'(t)). \label{eq:ide}
\end{equation}
DAEs naturally arise in many fields, e.g. power systems, chemical engineering, etc., as a result of modeling complex models. For example, in power systems, current-voltage relations describe the dynamics of circuit elements and represent differential equations. Kirchhoff's laws prescribe behaviors of flowing currents and voltages at nodes in a circuit. They are enforced through algebraic equations. Reaction-diffusion problems describe the interplay of reaction and diffusion between concentrations \cite{Benabdallah2025}, and algebraic equations are used to impose additional conditions, e.g., on the concentrations densities. Scientists are particularly interested in accurate and efficiently computed numerical solutions for such models. However, when solving DAEs, the mixture of numerical differentiation and integration poses a challenge to numerical solvers \cite{Ascher1998}. Since algebraic equations in semi-explicit DAEs arise as the stiff limit of a corresponding singular perturbation problem, explicit solvers are impractical, as stability constraints enforce small time step sizes. In contrast, implicit solvers are effective for DAEs, since they provide the desired order of accuracy and allow larger time step sizes due to their generous stability conditions. Stiffly accurate methods are preferred because they do not suffer from order reduction in algebraic variables. Common stiffly-accurate solvers for DAEs with high accuracy are Radau IIA methods that belong to the class of direct solvers with minimal order reduction. For index-one problems, order $2M - 1$ is achieved for $M$ collocation nodes in differential and algebraic variables. For higher-index problems, an order reduction in algebraic variables is expected \cite{Hairer_stiff2010}. Unlike diagonally implicit Runge-Kutta (DIRK) methods, for which full order cannot be expected in all variables even for index-one problems, differential variables retain the full order \cite{Brenan1996}. As a downside, the direct solve of the collocation problem is computationally expensive for Radau solvers because of the dense coefficient matrix.

Spectral deferred corrections (SDC) is a high-order iterative method developed by Dutt et al.\ \cite{Dutt2000}. It iteratively solves a series of error equations, and the current solution is corrected by adding the approximated error to the actual solution. In each iteration, the solution is computed via forward substitution where the solve of the system at each node has the same effort as computing a Euler step. For linear problems, SDC can be written as a preconditioned Richardson iteration \cite{Huang2006}. Each correction improves the order of the numerical solution by one, up to the maximum order of the underlying spectral quadrature rule \cite{Shu07}. In the last twenty-five years, a lot of work has been done on SDC. Splitting variants were developed for general ordinary differential equations (ODEs) \cite{Minion2003}, for fast- and slow wave problems \cite{Ruprecht2016}, and general IDEs \cite{Bu2012} to improve computational efficiency. The idea of SDC was expanded to a grid hierarchy, which led to the multilevel SDC. The convergence of the scheme has been examined in \cite{Kremling2021} and has already been applied to evolution problems \cite{Pfister2025}. The parallel framework Parallel-Full-Approximation-Scheme-in-Space-and-Time (PFASST) is built on multilevel SDC that enables \textit{parallelization across the steps} \cite{Emmett2012}. In \cite{Speck2018}, a numerical strategy was introduced to compute coefficients of preconditioners in diagonal form — used to \textit{parallelize SDC across the method} — was introduced, whereas \cite{Caklovic2025} proposes an analytical approach to obtain such coefficients. In \cite{Saupe2025}, steps towards SDC with adaptive step size were taken to make the simulation as computationally efficient as possible. Recent progress has been made in applying SDC to DAEs; Huang et al. \cite{Huang2007} introduced an SDC method for general IDEs \prettyref{eq:ide} that applies the correction procedure within the DAE. The unknown solution can be expressed in terms of its derivative by the fundamental theorem of calculus, which becomes the new unknown of the problem. The solution is then recovered by spectral integration. The authors also propose another SDC variant suited for semi-explicit DAEs where numerical integration is only restricted to differential variables. This is more efficient because of the missing differential relation of the algebraic variables \cite{Huang2007}. In \cite{Qu2016}, a comprehensive analysis of the different SDC formulations is performed. Another possibility of treating the problem is to split the equations into different parts, which allows an implicit-explicit (IMEX) treatment \cite{Bu2012}. As an important application in computational fluid dynamics, different time integration methods based on SDC are considered for the incompressible Navier-Stokes equations, which is a problem of index two \cite{MinionSaye2018}, \cite{Stiller2020}.

In this work, we propose an SDC method that exploits the semi-explicit structure of the DAE resulting in a more efficient method: it combines the traditional SDC approach for differential equations with the idea of enforcing algebraic constraints in each iteration. We theoretically show that the solution of the proposed scheme elevates one order per iteration up to the maximum order, and confirm the result in numerical experiments. As for traditional SDC, it is possible to parallelize the proposed scheme. In comparison with existent DAE solvers, we demonstrate its competitiveness at least when it is parallelized. The work is structured as follows. The traditional SDC method is derived in detail in \prettyref{sec:sdc}. In \prettyref{sec:sdc_dae}, we introduce existing SDC solvers for DAEs and derive a new scheme suited for semi-explicit DAEs that applies the SDC technique to differential equations while keeping algebraic equations as implicit condition. We also prove that the solution of the method gains one order per iteration. In \prettyref{sec:numerical_results}, we confirm the theoretical results and study the performance of the proposed scheme. We conclude the work in \prettyref{sec:conclusion}.

\section{Spectral deferred corrections} \label{sec:sdc}
The method of spectral deferred corrections was first developed by Dutt et al. \cite{Dutt2000}. In this section, we derive the method in detail to provide an understanding of how the scheme is constructed. Consider an initial-value problem for an ODE system for $t \in \mathcal{I} := [t_0, t_1]$ of the form
\begin{equation}
    \bm{u}'(t) = \bm{f}(\bm{u}(t), t), \quad \bm{u}(t_0) = \bm{u}_0 \label{eq:ode}
\end{equation}
with initial condition $\bm{u}_0: \mathcal{I} \to \mathbb{R}^n$, the unknown solution $\bm{u}: \mathcal{I} \to \mathbb{R}^n$, and the right-hand side of the ODE $\bm{f}:\mathbb{R}^n \times \mathcal{I} \to \mathbb{R}^n$. The interval length of $\mathcal{I}$ denotes the time step size as $\Delta t := t_1 - t_0$. Integrating the differential equation in \prettyref{eq:ode} over $\mathcal{I}$ we obtain Picard's integral formulation given by
\begin{equation}
    \bm{u}(t) = \bm{u}_0 + \int_{t_0}^t \bm{f}(\bm{u}(s), s)\,\mathrm{ds}. \label{eq:picard_integral}
\end{equation}
Let $t_0 \leq \tau_1 < .. <\tau_M \leq t_1$ be a set of $M$ collocation nodes with substeps $\Delta \tau_m := \tau_m - \tau_{m - 1}$ for $m = 2,\dots,M$ and $\Delta \tau_1 := \tau_1 - t_0$. In the following, Radau IIA nodes with $t_0 < \tau_1$ and $\tau_M = t_1$ are used. The integral equation is approximated by a spectral quadrature rule at each node $\tau_m$
\begin{equation}
    \bm{u}(\tau_m) = \bm{u}(t_0) + \sum_{j = 1}^M q_{m, j} \bm{f}(\bm{u}(\tau_j), \tau_j) \label{eq:collocation_node_wise}
\end{equation}
with quadrature weights
\begin{equation}
    q_{m, j} = \int_{t_0}^{\tau_m} \ell_j(s)\,\mathrm{ds}
\end{equation}
ensuring high accuracy. The function $\ell_j (t)$ denotes the $j$-th Lagrange polynomial
\begin{equation}
    \ell_j (t) = \prod_{i=1, j\neq i}^{M}\dfrac{t-\tau_{i}}{\tau_{i}-\tau_{j}}.
\end{equation}
In general, using the spectral quadrature rule with $M$ Radau IIA nodes, the integral over the entire interval $\mathcal{I}$ can be computed with error $\mathcal{O}(\Delta t^{2M -1})$. However, the integrals in Picard's formulation are computed with error $\mathcal{O}(\Delta t^{M + 1})$, since the integrals are only defined on the subintervals $[t_0, \tau_j]$ for $j = 1,\dots,M$.

Equations \prettyref{eq:collocation_node_wise} are equivalent to the stages in a general implicit Runge-Kutta (RK) method represented by the Butcher tableau
\begin{equation*}
    \renewcommand\arraystretch{1.6}
    \begin{array}{c@{\hskip 0.5em}|@{\hskip 0.5em}ccc}
        c_1 & q_{1, 1} & \cdots & q_{1,M}\\
        \vdots & \vdots & & \vdots \\
        c_M & q_{M, 1}&  \cdots & q_{M, M}\\[0.5ex]
        \hline
        & b_1 & \cdots & b_M 
    \end{array}
\end{equation*}
with weights $b_j$, $j=1,..,M$ and nodes $c_m \in [0, 1]$, $m=1,..,M$ where $\tau_m = t_0 + c_m \Delta t$.
Across all collocation nodes, the collocation problem is then given by
\begin{equation}
    \bm{u} = \bm{1}_M \otimes \bm{u}_0 + \Delta t \bm{Q} \otimes \bm{I}_n \bm{f} (\bm{u}) \label{eq:collocation_problem}
\end{equation}
with $\bm{1}_M := (1, \dots ,1)^\top \in \mathbb{R}^M$, the vector of the unknown function at collocation nodes $\bm{u} \allowbreak :=\allowbreak (\bm{u}(\tau_1),\allowbreak \dots,\allowbreak \bm{u}(\tau_M))^\top \in \mathbb{R}^{Mn}$, and the vector of the evaluations of the corresponding right-hand side $\bm{f} (\bm{u}) := \allowbreak (\bm{f}(\bm{u}(\tau_1), \tau_1),\allowbreak \dots,\allowbreak \bm{f}(\bm{u}(\tau_M), \tau_M))^\top \allowbreak \in \allowbreak \mathbb{R}^{Mn}$. The matrix $\bm{Q} = \{q_{m, j}\}_{m, j=1,\dots,M}$ denotes the spectral integration matrix and $\bm{I}_n$ is the identity matrix of size $n$. If the last collocation node does not equal the end of the time step, i.e. $\tau_M < t_1$, the solution at the next time $t_1$ is obtained by performing the collocation update
\begin{equation}
    \bm{u}(t_1) = \bm{u}_0 + \sum_{j = 1}^M b_j \bm{f}(\bm{u}(\tau_j), \tau_j).
\end{equation}

The implicit system \prettyref{eq:collocation_problem} defines a system of $Mn$ equations with $Mn$ unknowns, and therefore the computation of a solution is an expensive task, especially if $n$ is large. This is the case if the right-hand side stems from the spatial discretization of a partial differential equation, or \prettyref{eq:ode} defines a real-world application, for example.

Instead of directly solving the system, the SDC method iteratively solves a series of correction equations, and an improved solution for the next iteration is obtained by correcting the solution of the current approximation. The values at the collocation nodes are computed by forward substitution, so that the work at each node is similar to that of a Euler step. This is the original idea in the derivation of the method as in \cite{Dutt2000}.

Assume a provisional solution $\bm{u}^0(t)$ that is computed using a low-order time-stepping method. Let $\bm{u}^k(t)$ be an approximation of $\bm{u}(t)$ for some index $k \geq 0$, and the error to measure the accuracy of the approximation is defined as $\bm{\delta}^k (t):= \bm{u}(t) - \bm{u}^k(t)$ with $\bm{\delta}^k (t_0) = \bm{0}$ and $\bm{u}^k(t_0) = \bm{u}_0$. The unknown solution is replaced by the error, and Picard's formulation \prettyref{eq:picard_integral} becomes
\begin{equation}
    \bm{u}^k(t) + \bm{\delta}^k (t) = \bm{u}_0 + \int_{t_0}^t \bm{f}(\bm{u}^k(s) + \bm{\delta}^k (s), s)\,\mathrm{ds}. \label{eq:picard_integral_approximation}
\end{equation}
An equation for the error is received by
\begin{equation}
    \bm{\delta}^k (t) = \int_{t_0}^t \bm{f}(\bm{u}^k(s) + \bm{\delta}^k (s), s) - \bm{f}(\bm{u}^k(s), s)\,\mathrm{ds} + \bm{r}^k (t) \label{eq:error_equation}
\end{equation}
with residual function
\begin{equation}
    \bm{r}^k (t) = \bm{u}_0 + \int_{t_0}^t \bm{f}(\bm{u}^k(s), s)\,\mathrm{ds} - \bm{u}^k(t)
\end{equation}
that is used to monitor the convergence during the iteration process. Evaluating equation \prettyref{eq:error_equation} at $t = \tau_m$ and $t = t_0$, and taking the difference gives
\begin{equation}
    \begin{split}
        &\bm{\delta}^k (\tau_m) - \bm{\delta}^k (t_0) \\
        &\qquad= \int_{t_0}^{\tau_m} \bm{f}(\bm{u}^k(s) + \bm{\delta}^k (s), s) - \bm{f}(\bm{u}^k(s), s)\,\mathrm{ds} + \bm{r}^k (\tau_m) - \bm{r}^k (t_0).
        \label{eq:difference_errors}
    \end{split}
\end{equation}
For the discretization of \prettyref{eq:difference_errors}, the difference of the residual functions is the residual at $\tau_m$ itself, i.e.,
\begin{equation}
    \bm{r}^k (\tau_m) - \bm{r}^k (t_0) = \bm{u}_0 + \int_{t_0}^{\tau_m} \bm{f}(\bm{u}^k (s), s)\,\mathrm{ds} - \bm{u}^k (\tau_m) = \bm{r}^k (\tau_m). \label{eq:difference_residuals}
\end{equation}
If the residual $\bm{r}^{\tilde{k}} (\tau_m)$ is zero for any index $\tilde{k}$, the collocation problem is solved. In order to numerically compute the residual, the spectral quadrature rule is used to discretize the integral. Approximating the residual in \prettyref{eq:difference_residuals} by
\begin{equation}
    \bm{u}_0 + \sum_{j = 1}^M q_{m, j} \bm{f}(\bm{u}^k (\tau_j), \tau_j) - \bm{u}^k (\tau_m)
\end{equation}
and inserting it into \prettyref{eq:difference_errors}, the modified equation is
\begin{equation}
    \begin{split}
        \bm{u}^k (\tau_m) + \bm{\delta}^k (\tau_m) &= \bm{u}_0 + \int_{t_0}^{\tau_m} \bm{f}(\bm{u}^k(s) + \bm{\delta}^k (s), s) - \bm{f}(\bm{u}^k(s), s)\,\mathrm{ds} \\
        &\quad\,+ \sum_{j = 1}^M q_{m, j} \bm{f}(\bm{u}^k (\tau_j), \tau_j),
    \end{split} \label{eq:modified_error}
\end{equation}
where the error $\bm{\delta}^k (t_0)$ is zero. The integral in \prettyref{eq:modified_error} is simply discretized using either the left rectangle rule (as implicit Euler steps) by
\begin{equation}
    \begin{split}
        &\int_{t_0}^{\tau_m} \bm{f}(\bm{u}^k(s) + \bm{\delta}^k (s), s) - \bm{f}(\bm{u}^k(s), s)\,\mathrm{ds} \\
        & \qquad \qquad \qquad \approx \sum_{j = 1}^m \Delta \tau_j \left(\bm{f}(\bm{u}^k(\tau_j) + \bm{\delta}^k (\tau_j), \tau_j) - \bm{f}(\bm{u}^k(\tau_j), \tau_j)\right),
    \end{split} \label{eq:integral_left_rectangle_rule}
\end{equation}
or the right rectangle rule (as explicit Euler steps) by
\begin{equation}
    \begin{split}
        &\int_{t_0}^{\tau_m} \bm{f}(\bm{u}^k(s) + \bm{\delta}^k (s), s) - \bm{f}(\bm{u}^k(s), s)\,\mathrm{ds} \\
        & \qquad \qquad \qquad \approx \sum_{j = 1}^{m - 1} \Delta \tau_{j + 1} \left(\bm{f}(\bm{u}^k(\tau_j) + \bm{\delta}^k (\tau_j), \tau_j) - \bm{f}(\bm{u}^k(\tau_j), \tau_j)\right),
    \end{split} \label{eq:integral_right_rectangle_rule}
\end{equation}
where both quadrature rules are first-order. Let $\bm{u}^k_m \approx \bm{u}^k (\tau_m)$ be discrete approximations of the exact values. The solution is corrected by adding the error to the actual approximation, i.e. $\bm{u}^{k + 1}_m = \bm{u}^k_m + \bm{\delta}^k (\tau_m)$. Collecting the update equation \prettyref{eq:modified_error} with implicit Euler as the base integration method \prettyref{eq:integral_left_rectangle_rule}, the implicit SDC scheme suited for stiff problems reads
\begin{equation}
    \bm{u}^{k + 1}_m  = \bm{u}_0 + \sum_{j = 1}^m \Delta \tau_j \left(\bm{f}(\bm{u}^{k + 1}_j, \tau_j) - \bm{f}(\bm{u}^k_j, \tau_j)\right) + \sum_{j = 1}^M q_{m, j} \bm{f}(\bm{u}^k_j, \tau_j), \label{eq:impl_sdc}
\end{equation}
and the explicit SDC scheme using the explicit Euler as base integrator \prettyref{eq:integral_right_rectangle_rule} in \prettyref{eq:modified_error} is formulated by
\begin{equation}
    \bm{u}^{k + 1}_m  = \bm{u}_0 + \sum_{j = 1}^{m - 1} \Delta \tau_{j + 1} \left(\bm{f}(\bm{u}^{k + 1}_j, \tau_j) - \bm{f}(\bm{u}^k_j, \tau_j)\right) + \sum_{j = 1}^M q_{m, j} \bm{f}(\bm{u}^k_j, \tau_j), \label{eq:expl_sdc}
\end{equation}
rather suited for non-stiff problems.

The implicit scheme \prettyref{eq:impl_sdc} requires the solution of an implicit system at each collocation node. Since all $m - 1$ values $\bm{u}^{k + 1}_j$ are already computed, the solution of the system at node $\tau_m$ requires the same work as for one implicit Euler step. The same argument carries to the explicit scheme \prettyref{eq:expl_sdc}: Here, only the evaluation of the right-hand side is required to update the values which is just as cheap as an explicit Euler step.

Both SDC schemes have the general form
\begin{equation}
    \bm{u}^{k + 1}_m  = \bm{u}_0 + \sum_{j = 1}^m \tilde{q}_{m, j} \left(\bm{f}(\bm{u}^{k + 1}_j, \tau_j) - \bm{f}(\bm{u}^k_j, \tau_j)\right) + \sum_{j = 1}^M q_{m, j} \bm{f}(\bm{u}^k_j, \tau_j), \label{eq:general_sdc}
\end{equation}
where $\tilde{q}_{m, j}$ are the coefficients of a lower triangular matrix $\bm{Q}_\Delta$ associated with a low-order quadrature rule. In the community, it is well-known that the general SDC scheme using
\begin{equation}
    \bm{Q}_\Delta^{\texttt{IE}} = \begin{pmatrix}
        \Delta \tau_1 & 0 & \dots & 0\\
        \Delta \tau_1 & \Delta \tau_2 & \ddots & \vdots\\
        \vdots & \vdots & \ddots & 0 \\
        \Delta \tau_1 & \Delta \tau_2 & \dots & \Delta \tau_M
    \end{pmatrix} \quad \text{and} \quad
    \bm{Q}_\Delta^{\texttt{EE}} = \begin{pmatrix}
        0 & \dots & \dots & 0\\
        \Delta \tau_2 & 0 & & \vdots\\
        \vdots & \ddots & \ddots & \vdots\\
        \Delta \tau_2 & \dots & \Delta \tau_M & 0
    \end{pmatrix} \label{eq:QI_IE_EE}
\end{equation}
refers to the implicit scheme \prettyref{eq:impl_sdc} and the explicit scheme \prettyref{eq:expl_sdc}, respectively.

The traditional SDC method uses a low-order method to compute a provisional solution at each collocation node to obtain provisional values for $\bm{u}^0$. In the following, a provisional solution is used that is obtained instead by spreading the initial condition to each node $\tau_m$, i.e., $\bm{u}^0 \allowbreak := \allowbreak (\bm{u}_0,\allowbreak \dots,\allowbreak \bm{u}_0)^\top \in \mathbb{R}^{Mn}$.

\subsection{SDC as fixed-point method} \label{sec:sdc_iter_method}
The collocation problem \prettyref{eq:collocation_problem} as a system of equations can be solved by an iterative method. Equivalently, the system has the form
\begin{equation}
    \bm{C}(\bm{u}) = \bm{1}_M \otimes \bm{u}_0 \label{eq:collocation_problem_reformulated}
\end{equation}
with
\begin{equation}
    \quad\bm{C}(\bm{u}) := (\bm{I}_M \otimes \bm{I}_n - \Delta t \bm{Q} \otimes \bm{I}_n \bm{f}) (\bm{u}). \label{eq:operator_collocation_problem}
\end{equation}
If we apply an preconditioned iterative method of the form
\begin{equation}
    \bm{P}(\bm{u}^{k + 1}) = \bm{P}(\bm{u}^k) + (\bm{1}_M \otimes \bm{u}_0 - \bm{C}(\bm{u}^k))
\end{equation}
with preconditioner
\begin{equation}
    \quad\bm{P}(\bm{u}) := (\bm{I}_M \otimes \bm{I}_n - \Delta t \bm{Q}_\Delta \otimes \bm{I}_n \bm{f}) (\bm{u})
\end{equation}
we obtain the preconditioned Richardson iteration
\begin{equation}
    (\bm{I}_M \otimes \bm{I}_n - \Delta t \bm{Q}_\Delta \otimes \bm{I}_n \bm{f}) (\bm{u}^{k + 1}) = \bm{1}_M \otimes \bm{u}_0 + (\Delta t (\bm{Q} - \bm{Q}_\Delta) \otimes \bm{I}_n \bm{f}) (\bm{u}^k) \label{eq:sdc_iterative_scheme}
\end{equation}
with iteration matrix
\begin{equation}
    (\bm{I}_M \otimes \bm{I}_n - \Delta t \bm{Q}_\Delta \otimes \bm{I}_n \bm{f})^{-1}\Delta t (\bm{Q} - \bm{Q}_\Delta) \otimes \bm{I}_n \bm{f} (\bm{u}) \label{eq:iteration_matrix}
\end{equation}
denoted by $\bm{K} (\bm{u})$. The row-wise formulation of the scheme \prettyref{eq:sdc_iterative_scheme} corresponds to the implicit SDC method \prettyref{eq:impl_sdc}, or the explicit SDC scheme \prettyref{eq:expl_sdc}, depending on the choice of $\bm{Q}_\Delta$.

In numerical experiments, we focus on a certain set of $\mathbf{Q}_\Delta$ matrices including the low-order quadrature rules $\bm{Q}_\Delta^{\texttt{IE}}$ and $\bm{Q}_\Delta^{\texttt{EE}}$. The matrix $\bm{Q}_\Delta^{\texttt{Picard}} = \bm{0}$ yields an explicit scheme and equals the Picard iteration known to be suited for non-stiff problems. We have seen that the SDC method is equivalent to a preconditioned Richardson iteration. In the last ten years, some work focused on the construction of $\bm{Q}_\Delta$ matrices that minimize the spectral radius of the iteration matrix \prettyref{eq:iteration_matrix}. The research work started with the observation that the original SDC methods \prettyref{eq:impl_sdc} and \prettyref{eq:expl_sdc} converge slowly. "St. Martin's LU-trick" significantly improves the convergence rate of the SDC method, especially in the stiff and non-stiff limits \cite{Weiser2015}. The resulting preconditioner leads to zero spectral radius, and thus eliminates the observed convergence issues. The construction of $\bm{Q}_\Delta^{\texttt{LU}}$ is based on minimizing the spectral radius of the iteration matrix in the limits. The matrix is defined by $\bm{Q}_\Delta^{\texttt{LU}}=\bm{U}^T$ where $\bm{U}$ is computed from the LU decomposition $\bm{Q}^T=\bm{LU}$.

Recently, an analytical approach to compute coefficients for diagonal matrices $\bm{Q}_\Delta^{\texttt{MIN-SR-S}}$ and $\bm{Q}_\Delta^{\texttt{MIN-SR-NS}}$ that allows SDC to be parallelized across the method is presented \cite{Caklovic2025}. The approach aims to minimize the spectral radius by computing the coefficients to get a nilpotent iteration matrix. The \texttt{MIN-SR-FLEX} strategy that was introduced in \cite{Caklovic2025} results in a non-stationary SDC method, i.e., the corresponding matrix $\bm{Q}_\Delta$ changes after each iteration, is not considered here.

\section{Spectral deferred corrections for differential-algebraic equations} \label{sec:sdc_dae}
The idea of the SDC method can be extended to the class of DAEs.
Consider the initial-value problem for a semi-explicit DAE for $t \in \mathcal{I}$ of the form
\begin{equation}
    \bm{y}' (t) = \bm{f} (\bm{y}(t), \bm{z}(t), t), \quad \bm{0} = \bm{g} (\bm{y}(t), \bm{z}(t), t), \quad (\bm{y}(t_0), \bm{z}(t_0)) = (\bm{y}_0, \bm{z}_0) \label{eq:semiexplicit_dae}
\end{equation}
with initial conditions $\bm{y}_0: \mathcal{I} \to \mathbb{R}^{n_d}$, $\bm{z}_0: \mathcal{I} \to \mathbb{R}^{n_a}$, where $\bm{y}: \mathcal{I} \to \mathbb{R}^{n_d}$ is the differential variable, and $\bm{z}: \mathcal{I} \to \mathbb{R}^{n_a}$ is the algebraic variable. The function $\bm{f}:\mathbb{R}^{n_d} \times \mathbb{R}^{n_a} \times \mathcal{I} \to \mathbb{R}^{n_d}$ denotes the right-hand side of the differential equations, and $\bm{g}:\mathbb{R}^{n_d} \times \mathbb{R}^{n_a} \times \mathcal{I} \to \mathbb{R}^{n_a}$ is the right-hand side of the algebraic constraints. We set $n_d$ as the number of differential equations and differential variables, and $n_a$ as the number of algebraic constraints and algebraic variables with $n_d + n_a = n$ the size of the entire system \prettyref{eq:semiexplicit_dae}.
Let $\bm{J}_{\bm{f}}$ be the Jacobian of $\bm{f}$ with
\begin{equation}
    \bm{J}_{\bm{f}} = \begin{pmatrix}
        \frac{\partial \bm{f}_1}{\partial y_1} & \cdots & \frac{\partial \bm{f}_1}{\partial y_{n_d}} & \frac{\partial \bm{f}_1}{\partial z_1} & \cdots & \frac{\partial \bm{f}_1}{\partial z_{n_a}} \\
        \vdots & & \vdots & \vdots & &\vdots \\
        \frac{\partial \bm{f}_{n_d}}{\partial y_1} & \cdots & \frac{\partial \bm{f}_{n_d}}{\partial y_{n_d}} & \frac{\partial \bm{f}_{n_d}}{\partial z_1} & \cdots & \frac{\partial \bm{f}_{n_d}}{\partial z_{n_a}}
    \end{pmatrix} = \begin{pmatrix}
            \bm{J}_{\bm{f}}^{\bm{y}} & & \bm{J}_{\bm{f}}^{\bm{z}}
        \end{pmatrix},
\end{equation}
where $\bm{J}_{\bm{f}}^{\bm{y}} $ and $\bm{J}_{\bm{f}}^{\bm{z}}$ are Jacobian parts with respect to $\bm{y}$ and $\bm{z}$. The components of the right-hand side $\bm{f}$ are defined by $\bm{f}_i = \bm{f}_i (\bm{y} (t), \bm{z} (t), t)$ with $\bm{y} (t) = (y_1 (t), \dots, y_{n_d} (t))$ and $\bm{z} (t) = (z_1 (t), \dots, z_{n_a} (t))$. The Jacobian $\bm{J}_{\bm{g}}$ of $\bm{g}$ and its parts $\bm{J}_{\bm{g}}^{\bm{y}}$, $\bm{J}_{\bm{g}}^{\bm{z}}$ are defined in the same way.

In the remainder of this work, the system \prettyref{eq:semiexplicit_dae} is assumed to have index one, i.e., $\bm{J}_{\bm{g}}^{\bm{z}}$ is nonsingular and bounded. If $\bm{f}$ and $\bm{g}$ are at least $C^2$ functions, then the algebraic constraints can be solved to $\bm{z}$ in terms of a function $\bm{G}: \mathbb{R}^{n_d} \times \mathcal{I} \to \mathbb{R}^{n_a}$ implied by the implicit function theorem
\begin{equation}
    \bm{0} = \bm{g} (\bm{y}(t), \bm{z}(t), t) \quad \Leftrightarrow \quad \bm{z} (t) = \bm{G} (\bm{y}(t), t) \label{eq:ift}
\end{equation}
for fixed time $t$. Inserting the function $\bm{G}$ into the algebraic constraints, differentiating the left-hand side  via chain rule and solving to $\bm{J}_{\bm{G}}$ the Jacobian of $\bm{G}$ has the form
\begin{equation}
    \bm{J}_{\bm{G}} (\bm{y} (t), t) = -\left[\left(\bm{J}_{\bm{g}}^{\bm{z}}\right)^{-1} \bm{J}_{\bm{g}}^{\bm{y}} \right] (\bm{y}, \bm{G}(\bm{y} (t), t), t). \label{eq:jacobian_G_ift}
\end{equation}
If the arguments of the Jacobian matrices are clear, they will be omitted from now on for better readability.
 
Under the assumptions of the implicit function theorem, the semi-explicit index-one DAE \prettyref{eq:semiexplicit_dae} admits a locally unique solution expressed through $\bm{G}$. This representation allows us to reformulate the DAE as an IDE, for which we now introduce the SDC scheme, first proposed by J. Huang et al. \cite{Huang2007}. Consider the general IDE \prettyref{eq:ide}. Since a Picard formulation is not available or at least difficult to extract, the function $\bm{u}(t)$ is represented via the fundamental theorem of calculus as
\begin{equation*}
    \bm{u} (t) = \bm{u}_0 + \int_{t_0}^t \bm{U}(s)\,\mathrm{ds},
\end{equation*}
where $\bm{U} (t) := \bm{u}'(t)$ denotes the derivative of $\bm{u}$, and the IDE becomes
\begin{equation*}
    \bm{0} = \bm{F}\left(t, \bm{u}_0 + \int_{t_0}^t \bm{U}(s)\,\mathrm{ds}, \bm{U} (t)\right).
\end{equation*}
Then, the proposed fully-integrating SDC method is formulated as
\begin{equation}
    \bm{0} = \bm{F}\left(\tau_m, \bm{u}_0 + \sum_{j = 1}^M (q_{m, j} - \tilde{q}_{m, j}) \bm{U}^k_j + \sum_{j = 1}^m \tilde{q}_{m, j} \bm{U}^{k + 1}_j, \bm{U}^{k + 1}_m\right) \label{eq:fi_sdc}
\end{equation}
with approximations to the derivative at node $\tau_m$, i.e., $\bm{U}^k_m \approx \bm{u}'(\tau_m)$. In this work, the scheme \prettyref{eq:fi_sdc} is named \texttt{FI-SDC}. The implicit system at each node is solved using the Newton-Krylov method. After each iteration $k$, the computation of
\begin{equation*}
    \bm{u}^{k + 1}_m = \bm{u}_0 + \sum_{j = 1}^M q_{m, j} \bm{U}^{k + 1}_j
\end{equation*}
recovers the numerical solution $\bm{u}^{k + 1}_m$ by numerical integration. For semi-explicit DAEs, J. Huang et al. proposed a semi-integrating SDC variant \cite{Huang2007}
\begin{subequations} \label{eq:si_sdc}
    \begin{align}
        \bm{Y}^{k + 1}_m &= \bm{f}\left(\bm{y}_0 + \sum_{j = 1}^M (q_{m, j} - \tilde{q}_{m, j}) \bm{Y}^k_j + \sum_{j = 1}^m \tilde{q}_{m, j} \bm{Y}^{k + 1}_j, \bm{z}^{k + 1}_m\right),\\
        \bm{0} &= \bm{g}\left(\bm{y}_0 + \sum_{j = 1}^M (q_{m, j} - \tilde{q}_{m, j}) \bm{Y}^k_j + \sum_{j = 1}^m \tilde{q}_{m, j} \bm{Y}^{k + 1}_j, \bm{z}^{k + 1}_m\right),
    \end{align}
\end{subequations}
where $\bm{Y}^k_m$ are approximations to the derivative of the differential variable at node $\tau_m$, i.e., $\bm{Y}^k_m \approx \bm{y}'(\tau_m)$. The vectors $\bm{z}_m^k$ are approximations to the algebraic variable with $\bm{z}_m^k \approx \bm{z} (\tau_m)$. We denote the scheme \prettyref{eq:si_sdc} as \texttt{SI-SDC} where numerical integration is restricted to differential variables. For more details, the interested reader is referred to \cite{Huang2007}. 

In order to derive an SDC method for semi-explicit DAEs, we choose a different approach: The SDC method \prettyref{eq:general_sdc} is applied to differential equations subject to algebraic constraints. Then, the constrained SDC method denoted as \texttt{SDC-C} reads
\begin{subequations} \label{eq:constrained_sdc}
    \begin{align}
        \bm{y}^{k + 1}_m  &= \bm{y}_0 + \sum_{j = 1}^m \tilde{q}_{m, j} \left(\bm{f}(\bm{y}^{k + 1}_j, \bm{z}^{k + 1}_j, \tau_j) - \bm{f}(\bm{y}^k_j, \bm{z}^k_j, \tau_j)\right) \notag \\
        &\quad\,+ \sum_{j = 1}^M q_{m, j} \bm{f}(\bm{y}^k_j, \bm{z}^k_j, \tau_j), \label{eq:sdc_c_eq1} \\
        \bm{0} &= \bm{g} (\bm{y}^{k + 1}_m, \bm{z}^{k + 1}_m, \tau_m) \label{eq:sdc_c_eq2}
    \end{align}
\end{subequations}
with approximations $\bm{y}^k_m \approx \bm{y} (\tau_m)$ and $\bm{z}^k_m \approx \bm{z} (\tau_m)$. The derivation of the \texttt{SDC-C} scheme is based on the constrained Picard formulation
\begin{equation}
    \bm{y}(t) = \bm{y}_0 + \int_{t_0}^t \bm{f}(\bm{y}(s), \bm{z}(s), s)\,\mathrm{ds} \quad \text{s. t.} \quad \bm{0} = \bm{g} (\bm{y}(t), \bm{z} (t), t). \label{eq:constrained_picard_integral}
\end{equation}
Note that in the linear case the identity matrix $\bm{I}_n$ in the iteration matrix $\bm{K}$ in \prettyref{eq:iteration_matrix} is replaced by a matrix $\bm{I}_{n_d,0}$ whose first $n_d$ diagonal entries are equal to one and the remaining $n_a$ entries are zero.

\begin{remark}[Algebraic constraints as stiff limit] \label{rem:alg_stiff_limit}
    Let us study the scalar differential equation
    \begin{equation}
        \varepsilon \bar{z}'(t) = \bar{z}(t), \quad \bar{z}(t_0) = \bar{z}_0, \label{eq:scalar_equation}
    \end{equation}
    for an initial condition $\bar{z}_0 \in \mathbb{R}$, unknown solution $\bar{z} (t) \in \mathbb{R}$, and perturbation parameter $0 < \varepsilon \ll 1$ where $\varepsilon$ introduces stiffness to the problem. We apply an implicit SDC scheme that reads across all collocation nodes
    \begin{equation*}
        (\varepsilon \bm{I}_M - \Delta t \bm{Q}_\Delta) \bm{\bar{z}}^{k + 1} = \varepsilon \bm{1}_M \otimes \bar{z}_0 + \Delta t (\bm{Q} - \bm{Q}_\Delta) \bm{\bar{z}}^k
    \end{equation*}
    with solution
    \begin{equation*}
        \bm{\bar{z}}^{k + 1} = (\varepsilon \bm{I}_M - \Delta t \bm{Q}_\Delta)^{-1}(\varepsilon \bm{1}_M \otimes \bar{z}_0 + \Delta t (\bm{Q} - \bm{Q}_\Delta) \bm{\bar{z}}^k),
    \end{equation*}
    where $\bm{\bar{z}}^k \allowbreak :=\allowbreak (\bar{z}^k_1,\allowbreak \dots,\allowbreak \bar{z}^k_M)^\top \in \mathbb{R}^{M}$ with $\bar{z}^k_M \approx \bar{z}(\tau_m)$. Setting $\varepsilon = 0$ in \prettyref{eq:scalar_equation} that represents the stiff limit leads to a purely algebraic equation, and the solution of the modified scheme for the algebraic equation is
    \begin{equation*}
        \bm{\bar{z}}^{k + 1} = (\bm{I}_M - \bm{Q}_\Delta^{-1} \bm{Q}) \bm{\bar{z}}^k.
    \end{equation*}
    This highlights two key aspects: First, in a semi-explicit DAE \prettyref{eq:semiexplicit_dae} the algebraic constraints correspond to the stiff limit of the problem that requires their implicit treatment. The matrix $(\bm{I}_M - \bm{Q}_\Delta^{-1} \bm{Q})$ is well-known in the community as representing the stiff limit of the iteration matrix, and certain choices of $\bm{Q}_\Delta$ matrices are specifically designed for this case (see the introduction of the matrices at the end of \prettyref{sec:sdc_iter_method}). Second, since the purely algebraic equation does not involve any time derivatives,
    numerical integration is not required for its discretization. As a result, applying \texttt{FI-SDC} to a semi-explicit DAE, i.e., one that includes the numerical integration of the algebraic equations, yields an unnecessarily inefficient and inaccurate approach. Moreover, the algebraic equations are not guaranteed to be satisfied, as the method effectively solves a collocation problem rather than enforcing the constraints directly. This improper treatment may also cause the method to become unstable more quickly. While the severity of these issues is problem-dependent, the numerical integration of the algebraic components is generally not recommended in this context.
\end{remark}

For the \texttt{FI-SDC} and \texttt{SI-SDC} methods introduced above, the following remark shows that the algebraic equations only converge to zero instead of being zero in every iteration.

\begin{remark}[Convergence in algebraic constraints] \label{rem:convergence_alg_const}
    For a scalar problem, consider the \texttt{SI-SDC} scheme \prettyref{eq:si_sdc} across all collocation nodes
    \begin{equation} \label{eq:si_sdc_all_nodes}
        \begin{split}
            \bm{Y}^{k + 1} &= \bm{f}\left(\bm{y}_0 + \Delta t (\bm{Q} - \bm{Q}_\Delta)\bm{Y}^k + \Delta t \bm{Q}_\Delta \bm{Y}^{k + 1}, \bm{z}^{k + 1}, \bm{\tau}\right), \\
            \bm{0} &= \bm{g}\left(\bm{y}_0 + \Delta t (\bm{Q} - \bm{Q}_\Delta)\bm{Y}^k + \Delta t \bm{Q}_\Delta \bm{Y}^{k + 1}, \bm{z}^{k + 1}, \bm{\tau}\right),
        \end{split}
    \end{equation}
    with approximation vectors $\bm{Y}^k := (\bm{Y}^k_1,..,\bm{Y}^k_M)$ and $\bm{z}^k := (\bm{z}^k_1,..,\bm{z}^k_M)$, and the vector of collocation nodes $\bm{\tau} := (\tau_1, .., \tau_M)$. Assume that the solution of \prettyref{eq:si_sdc_all_nodes} is converged in iteration $\tilde{k} \le \kmax$, i.e., $\bm{Y}^{\tilde{k}} \approx \bm{Y}^{\tilde{k} + 1}$. We define converged solutions by $\bm{Y}$ and $\bm{z}$. Then, the solutions solve the corresponding collocation problem. Especially, the algebraic constraints in the \texttt{SI-SDC} scheme become
    \begin{equation}
        \bm{0} = \bm{g}\left(\bm{y}_0 + \Delta t (\bm{Q} - \bm{Q}_\Delta)\bm{Y} + \Delta t \bm{Q}_\Delta \bm{Y}, \bm{z}, \bm{\tau}) = \bm{g}(\bm{y}_0 + \Delta t \bm{Q} \bm{Y}, \bm{z}, \bm{\tau}\right). \label{eq:coll_problem_g}
    \end{equation}
    The solution $\bm{y}$ is recovered by numerical integration
    \begin{equation*}
        \bm{y} = \bm{y}_0 + \Delta t \bm{Q} \bm{Y}. 
    \end{equation*}
    Using this in \prettyref{eq:coll_problem_g}, we obtain
    \begin{equation}
        \bm{0} = \bm{g}\left(\bm{y}_0 + \Delta t \bm{Q} \frac{1}{\Delta t} \bm{Q}^{-1}(\bm{y} - \bm{y}_0), \bm{z}, \bm{\tau}\right) = \bm{g}(\bm{y}, \bm{z}, \bm{\tau}), \label{eq:coll_problem_g_2}
    \end{equation}
    and the same holds for the \texttt{FI-SDC} scheme. Equation \prettyref{eq:coll_problem_g_2} shows that algebraic constraints might not be satisfied in each iteration, but only if the problem is converged. Thus, the formulations of the \texttt{SI-SDC} and \texttt{FI-SDC} schemes do not guarantee the preservation of the conditions during the iteration process. In contrast, the \texttt{SDC-C} method ensures that the algebraic equations are satisfied in each iteration by keeping $\bm{0} = \bm{g}(\bm{y}^{k+1}, \bm{z}^{k+1}, \bm{\tau})$ as an implicit condition of the system, see \prettyref{eq:sdc_c_eq2}.
\end{remark}

Algebraic constraints represent the stiff limit in \prettyref{eq:semiexplicit_dae} as a stiff component. Consider the problem where the right-hand side of the differential equations can be split into parts in the form
\begin{equation}
    \begin{split}
        \bm{y}' (t) &= \bm{f}(\bm{y}(t), \bm{z}(t), t) = \bm{f}_\mathrm{im}(\bm{y}(t), \bm{z}(t), t) + \bm{f}_\mathrm{ex}(\bm{y}(t), \bm{z}(t), t), \\\bm{0} &= \bm{g}(\bm{y}(t), \bm{z}(t), t),
    \end{split} \label{eq:semi_explicit_dae_imex}
\end{equation}
where $\bm{f}_\mathrm{im}$ denotes the stiff term that is treated implicitly and $\bm{f}_\mathrm{ex}$ is the non-stiff term that is treated explicitly. Integrating the differential equations over $\mathcal{I}$, the constrained Picard integral equation has the form
\begin{equation}
    \begin{split}
        \bm{y}(t) &= \bm{y}_0 + \int_{t_0}^t \bm{f}_\mathrm{im} (\bm{y}(s), \bm{z}(s), s) + \bm{f}_\mathrm{ex} (\bm{y}(s), \bm{z}(s), s)\,\mathrm{ds}, \\
    \text{s. t.} \quad \bm{0} &= \bm{g} (\bm{y}(t), \bm{z} (t), t).
    \end{split} \label{eq:imex_constrained_picard_integral}
\end{equation}
The \texttt{SDC-C} method can be extended to problems of the form \prettyref{eq:semi_explicit_dae_imex} by treating the differential equations in an IMEX way. The resulting scheme, denoted \texttt{IMEX-SDC-C}, is formulated as
\begin{subequations}
    \begin{align}
        \bm{y}^{k + 1}_m  &= \bm{y}_0 + \sum_{j = 1}^m \tilde{q}_{m, j}^\mathrm{im} \left(\bm{f}_\mathrm{im}(\bm{y}^{k + 1}_j, \bm{z}^{k + 1}_j, \tau_j) - \bm{f}_\mathrm{im}(\bm{y}^k_j, \bm{z}^k_j, \tau_j)\right) \notag \\
        &\quad\,+ \sum_{j = 1}^m \tilde{q}_{m, j}^\mathrm{ex} \left(\bm{f}_\mathrm{ex}(\bm{y}^{k + 1}_j, \bm{z}^{k + 1}_j, \tau_j) - \bm{f}_\mathrm{ex}(\bm{y}^k_j, \bm{z}^k_j, \tau_j)\right) \notag \\
        &\quad\,+ \sum_{j = 1}^M q_{m, j} \bm{f}(\bm{y}^k_j, \bm{z}^k_j, \tau_j), \\
        \bm{0} &= \bm{g} (\bm{y}^{k + 1}_m, \bm{z}^{k + 1}_m, \tau_m),
    \end{align} \label{eq:imex_sdc_c}
\end{subequations}
with coefficients $\{\tilde{q}_{m, j}^\mathrm{im}\}_{j,m=1,..,M}$ of an implicit quadrature rule $\bm{Q}_\Delta^{\texttt{im}}$ and coefficients $\{\tilde{q}_{m, j}^\mathrm{ex}\}_{j,m=1,..,M}$ of an explicit quadrature rule $\bm{Q}_\Delta^{\texttt{ex}}$.

In contrast, a semi-implicit version for general IDEs is proposed in \cite{Bu2012}. They suggested splitting the right-hand side of the IDE to solve the problem in an IMEX way. In contrast to their approach, which applies SDC to the entire IDE - including the algebraic constraints - by splitting the right-hand side into implicit and explicit parts, our proposed \texttt{SDC-C} (and \texttt{IMEX-SDC-C}) restricts the spectral integration to the differential equations only and retains the algebraic constraint as an implicit condition in the system.

\subsection{Local truncation error of constrained SDC method} \label{sec:sdc_dae_lte}
Y. Xia et al. have shown that each iteration of the original SDC method \prettyref{eq:general_sdc} improves the order of the numerical solution by one order up to the maximum order of the underlying quadrature rule \cite{Shu07}. We modify the proof and show that the same result is obtained for the \texttt{SDC-C} method. The proof uses an induction (and resembles the proof structure in \cite[Lemma~2.1]{Shu07}). It includes three steps: First, we will show that the provisional solution across all collocation nodes obtained by spreading the initial condition to each node is first-order accurate, which corresponds to iteration $k = 0$. This differs from the approach to compute the provisional solution using a low-order method. In the second step, numerical solutions $\bm{y}^k_1$ and $\bm{z}^k_1$ at the first collocation node are shown to be of order $k + 1$. The first two steps represent the base case in \prettyref{thm:lte_sdc_c} that states the main result: The solutions $\bm{y}^k_m$ and $\bm{z}^k_m$ at arbitrary node $\tau_m$ are of order $\mathcal{O}(\Delta t^{k + 1})$. The entire structure of the proof is illustrated in \prettyref{fig:fig_1}.

The proof requires some preliminaries that are formulated below. The quadrature weights of the spectral quadrature rule and those of the low-order rule scale linearly with the time step size $\Delta t$.

\begin{figure}[!t]
    \centering
    \includegraphics[width=0.9\textwidth]{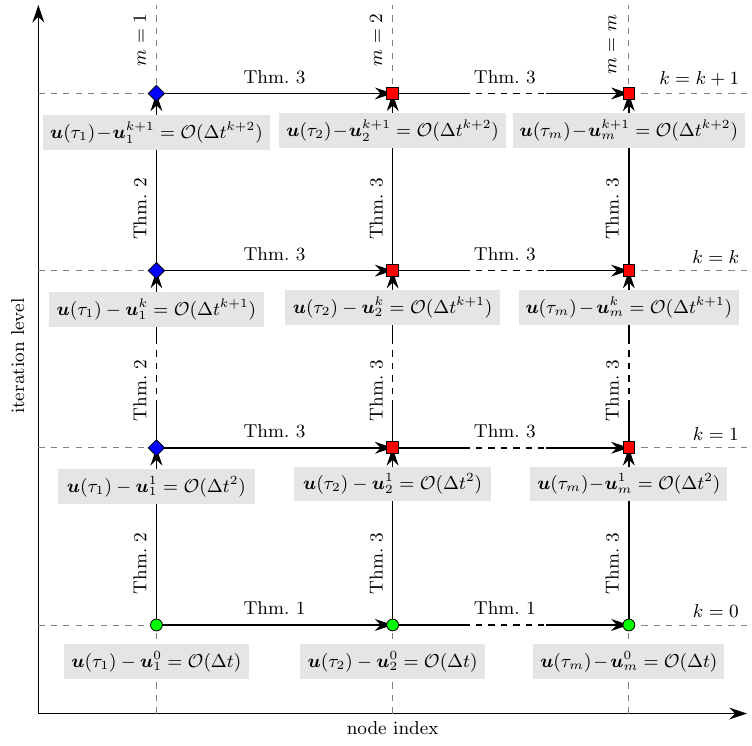}
    \caption{Structure of induction proof. Green nodes indicate the cases shown by \prettyref{thm:lte_initial_SDC_C} and present the base case for $k = 0$ and for all $m$. Blue nodes present cases proven by \prettyref{thm:lte_iters_first}, they represent the base case for $m = 1$ and all iteration levels $k$. The red nodes are obtained finally by \prettyref{thm:lte_sdc_c}.}
    \label{fig:fig_1}
\end{figure}

\begin{lemma}
    \label{lm:Q_IE_EE}
    The coefficients $\tilde{q}_{m, j}$ of the matrices $\bm{Q}_\Delta^{\texttt{IE}}$ and $\bm{Q}_\Delta^{\texttt{EE}}$ satisfy
    \begin{equation}
        \tilde{q}_{m, j} = \mathcal{O}(\Delta t)
    \end{equation}
    for $j, m = 1,..,M$.
\end{lemma}

\begin{proof}
    Consider \prettyref{eq:QI_IE_EE}.
\end{proof}

\begin{lemma}
    \label{lm:quad}
    Assume that the quadrature weights $\hat{q}_{m, j}$ are the result of spectral quadrature to integrate a function over the interval $[-1, 1]$. Then, the weights $q_{m, j}$ for integration over $[t_0, \tau_m]$ satisfy
    \begin{equation}
        q_{m, j} = \mathcal{O} (\Delta t).
    \end{equation}
\end{lemma}

\begin{proof}
    Note that $\hat{q}_{m, j}$ are independent of $\Delta t$. To integrate a function over $[t_0, \tau_m]$ a change of variables is applied. The transformed quadrature weights are
    \begin{equation}
        q_{m, j} = \frac{\tau_m - t_0}{2} \hat{q}_{m, j} = \mathcal{O} (\Delta t).
    \end{equation}
\end{proof}

For a compact interval $\mathcal{I}$, the boundedness of Jacobian matrices $\bm{J}_{\bm{G}}$, $\bm{J}_{\bm{f}}^{\bm{y}}$, and $\bm{J}_{\bm{f}}^{\bm{z}}$ is ensured, thus the matrices are in $\mathcal{O} (1)$.

\begin{lemma}
    \label{lm:compact_image}
    Let $\mathcal{I}$ be a compact interval. Let $\bm{y}$ and $\bm{z}$ be continuous functions on that compact interval. Then, the subset
    \begin{align*}
        \{t \in \mathcal{I}:(\bm{y}(t), \bm{z}(t)) \in \mathbb{R}^{n_d} \times \mathbb{R}^{n_a}\} \subset \mathbb{R}^{n_d} \times \mathbb{R}^{n_a}
    \end{align*}
    defines a compact subset.
\end{lemma}

\begin{proof}
    Since $\bm{y}$ and $\bm{z}$ are continuous in a compact interval, they are bounded by the extreme value theorem.
\end{proof}

\begin{lemma}
    \label{lm:bounded_J_G}
    Assume $\bm{g}$ is at least a $C^2$ function, i.e., $\bm{g}$ is differentiable and the partial derivatives $\frac{\partial \bm{g}_i}{\partial y_k}$ and $\frac{\partial \bm{g}_i}{\partial z_j}$ for $i,j=1,..,n_a$, $k=1,..,n_d$ exist and are continuous. Additionally, we assume that $\bm{J}_{\bm{g}}^{\bm{z}}$ is invertible. Then, the Jacobian $\bm{J}_{\bm{G}}$ of the function $\bm{G}$ in \prettyref{eq:jacobian_G_ift} is bounded.
\end{lemma}

\begin{proof}
    Applying \prettyref{lm:compact_image} the domain of $\bm{g}$ denoted as $D := \{(\bm{y}(t), \bm{z}(t), t): t \in \mathcal{I}\}$ is a compact subset in $\mathbb{R}^{n_d} \times \mathbb{R}^{n_a} \times \mathcal{I}$. Therefore, the image $\bm{g}(D)$ of a continuous function with a compact domain $D$ is compact. Since $\bm{g}$ is at least continuously differentiable, the partial derivatives exist and therefore are bounded on $D$.

    Differentiating the left-hand side of the equation $\bm{g}(\bm{y} (t), \bm{G}(\bm{y} (t), t), t) = \bm{0}$ using the chain rule and solving to $\bm{J}_{\bm{G}}$ the Jacobian of $\bm{G}$ is given by \prettyref{eq:jacobian_G_ift}.
    Due to the invertibility of $\bm{J}_{\bm{g}}^{\bm{z}}$, both the inverse of $\bm{J}_{\bm{g}}^{\bm{z}}$ and $\bm{J}_{\bm{g}}^{\bm{y}}$ are bounded, and thus $\bm{J}_{\bm{G}}$. 
\end{proof}

\begin{lemma}
    \label{lm:bounded_J_f}
    Assume $\bm{f}$ is at least a $C^2$ function, i.e., $\bm{f}$ is differentiable and the partial derivatives $\frac{\partial \bm{f}_i}{\partial y_j}$ and $\frac{\partial \bm{f}_i}{\partial z_k}$ for $i,j=1,..,n_d$, $k=1,..,n_a$ exist and are continuous. Then, the Jacobian matrices $\bm{J}_{\bm{f}}^{\bm{y}}$ and $\bm{J}_{\bm{f}}^{\bm{z}}$ are bounded.
\end{lemma}

\begin{proof}
    Applying \prettyref{lm:compact_image} (as in the proof for \prettyref{lm:bounded_J_G}) already proves what is to be shown. 
\end{proof}

Define the local errors in $\bm{y}$ and $\bm{z}$ at node $\tau_m$ in iteration $k$ as $\bm{e}^k_{\bm{y},m} := \bm{y}(\tau_m) - \bm{y}^k_m$ and $\bm{e}^k_{\bm{z},m} := \bm{z}(\tau_m) - \bm{z}^k_m$. For the provisional solution, we use the initialization
\begin{equation}
    \bm{y}^0_m = \bm{y}_0 \qquad \text{and} \qquad \bm{z}^0_m = \bm{z}_0,
\end{equation}
that is obtained by spreading the initial condition to each node $\tau_m$. For semi-explicit DAE problems, the provisional solution should be consistent with the algebraic constraints. The following result states that the initialization is first-order accurate.

\begin{theorem}
    \label{thm:lte_initial_SDC_C}
    Let $\bm{f}$ and $\bm{g}$ be at least $C^2$ functions. The problem \prettyref{eq:semiexplicit_dae} is assumed to have index one, i.e., the Jacobian part $\bm{J}_{\bm{g}}^{\bm{z}}$ is invertible and is assumed to be bounded. Then, the provisional solution of the \texttt{SDC-C} scheme satisfies
    \begin{equation}
        \bm{e}^0_{\bm{y},m} = \mathcal{O}(\Delta t) \quad \text{and} \quad \bm{e}^0_{\bm{z},m} = \mathcal{O}(\Delta t)
    \end{equation}
    for $m=1, \dots ,M$.
\end{theorem}

\begin{proof}
    Subtracting the provisional solution in $\bm{y}$ from the evaluated Picard integral equation \prettyref{eq:constrained_picard_integral} at $t = \tau_m$ we have
    \begin{equation}
    \begin{split}
        \bm{e}^0_{\bm{y},m} &= \bm{y}(t_0) + \int_{t_0}^{\tau_m} \bm{f}(\bm{y}(s), \bm{z}(s), s)\,\mathrm{ds} - \bm{y}_0 = \mathcal{O}(\Delta t)\\
    \end{split} \label{eq:LTE_y_provisional}
\end{equation}
by applying the left rectangle rule. Using \prettyref{eq:ift} the local error in $\bm{z}$ is given by
\begin{equation}
    \bm{e}^0_{\bm{z},m} = \bm{G}(\bm{y}(\tau_m), \tau_m) - \bm{G}(\bm{y}^0_m, \tau_m).
\end{equation}
Expanding $\bm{G}(\bm{y}^0_m, \tau_m)$ in a Taylor series around $(\bm{y} (t), t) = (\bm{y}(\tau_m), \tau_m)$ by
\begin{equation}
    \begin{split}
        \bm{G}(\bm{y}^0_m, \tau_m) 
        &= \bm{G}(\bm{y}(\tau_m), \tau_m) -\bm{J}_{\bm{G}} (\bm{y}(\tau_m), \tau_m) \bm{e}^0_{\bm{y},m} + \mathcal{O}((\bm{e}^0_{\bm{y},m})^2),
    \end{split}
\end{equation}
the error becomes
\begin{equation*}
    \begin{split}
        \bm{G}(\bm{y}(\tau_m), \tau_m) - \bm{G}(\bm{y}^0_m, \tau_m) 
        &= \bm{J}_{\bm{G}} (\bm{y}(\tau_m), \tau_m) \bm{e}^0_{\bm{y},m} + \mathcal{O}((\bm{e}^0_{\bm{y},m})^2) = \mathcal{O}(\Delta t)\\
    \end{split}
\end{equation*}
using \prettyref{eq:LTE_y_provisional} and \prettyref{lm:bounded_J_G}.
\end{proof}

\begin{remark}
    The original SDC method in \cite{Dutt2000} uses a provisional solution computed by implicit Euler, or explicit Euler that is different to the choice in this work as also noted in \prettyref{sec:sdc}. Nevertheless, spreading the initial condition provides a solution with first-order accuracy as shown above.
\end{remark}

\begin{theorem}
    \label{thm:lte_iters_first}
    Under the assumptions of \prettyref{thm:lte_initial_SDC_C} the numerical solutions in iteration $k$ at first node $\tau_1$ obtained by the \texttt{SDC-C} scheme satisfy
    \begin{equation}
        \bm{e}^k_{\bm{y},1} = \mathcal{O}(\Delta t^{k + 1}) \quad \text{and} \quad \bm{e}^k_{\bm{z},1} = \mathcal{O}(\Delta t^{k + 1})
    \end{equation}
    for $k=0,1,..,2M-1$.
\end{theorem}

\begin{proof}[by induction]
    In the base case, the error $\bm{e}^1_{\bm{z},1}$ of \texttt{SDC-C} can easily be determined by computing the difference
    \begin{equation}
        \begin{split}
            \bm{e}^1_{\bm{z},1} &= \bm{G}(\bm{y}(\tau_1), \tau_1) - \bm{G}(\bm{y}^1_1, \tau_1) 
            = \bm{J}_{\bm{G}} (\bm{y}(\tau_1), \tau_1) \bm{e}^1_{\bm{y},1} + \mathcal{O}((\bm{e}^1_{\bm{y},1})^2), \label{eq:LTE_z}
        \end{split}
    \end{equation}
    where $\bm{G }(\bm{y}^1_1, \tau_1)$ is expanded in a Taylor series around $(\bm{y} (t), t) = (\bm{y}(\tau_1), \tau_1)$ by
    \begin{equation*}
        \begin{split}
            \bm{G}(\bm{y}^1_1, \tau_1) 
            &= \bm{G}(\bm{y}(\tau_1), \tau_1) - \bm{J}_{\bm{G}} (\bm{y}(\tau_1), \tau_1) \bm{e}^1_{\bm{y},1} + \mathcal{O}((\bm{e}^1_{\bm{y},1})^2).
        \end{split}
    \end{equation*}
    Thus, the local error in $\bm{z}$ depends on the local error in $\bm{y}$. We now determine the local error in $\bm{y}$. Taking the difference of the evaluated Picard integral equation \prettyref{eq:constrained_picard_integral} at $t = \tau_1$ and equation \prettyref{eq:sdc_c_eq1} for $k = 1$ leads to
    \begin{equation}
        \begin{split}
            \bm{e}^1_{\bm{y},1} 
            &= \int_{t_0}^{\tau_1} \bm{f}(\bm{y}(s), \bm{z}(s), s)\,\mathrm{ds} - \sum_{j=1}^M q_{1,j} \bm{f}(\bm{y}^0_j, \bm{z}^0_j, \tau_j) \\
            &\quad- \tilde{q}_{1,1} (f(\bm{y}^1_1, \bm{z}^1_1, \tau_1) - \bm{f}(\bm{y}^0_1, \bm{z}^0_1, \tau_1)). \label{eq:base_eq}
        \end{split}
    \end{equation}
    To estimate the spectral quadrature error in \prettyref{eq:base_eq} adding a zero term, we obtain
    \begin{subequations} \label{eq:first_quantity}
        \begin{align}
            &\int_{t_0}^{\tau_1} \bm{f}(\bm{y}(s), \bm{z}(s), s)\,\mathrm{ds} - \sum_{j=1}^M q_{1,j} \bm{f}(\bm{y}^0_j, \bm{z}^0_j, \tau_j) \label{eq:first_quantity_eq1} \\
            &\qquad\qquad= \left(\int_{t_0}^{\tau_1} \bm{f}(\bm{y}(s), \bm{z}(s), s)\,\mathrm{ds} - \sum_{j=1}^M q_{1,j} \bm{f}(\bm{y}(\tau_j), \bm{z}(\tau_j), \tau_j)\right) \label{eq:first_quantity_eq2} \\
            &\qquad\qquad\quad+ \left(\sum_{j=1}^M q_{1,j} \bm{f}(\bm{y}(\tau_j), \bm{z}(\tau_j), \tau_j) - \sum_{j=1}^M q_{1,j} \bm{f}(\bm{y}^0_j, \bm{z}^0_j, \tau_j)\right). \label{eq:first_quantity_eq3}
        \end{align}
    \end{subequations}
    Since the quadrature rule with weights $q_{1, j}$ only requires the values of $\bm{f}$ evaluated at the collocation nodes, we can interpret $\bm{f}$ evaluated at $(\bm{y}(\tau_j), \bm{z}(\tau_j), \tau_j)$ for $j=1,..,M$ as an interpolating polynomial with error $\mathcal{O}(\Delta t^M)$. Using \prettyref{lm:quad}, the spectral quadrature error in \prettyref{eq:first_quantity_eq2} is estimated as
    \begin{equation*}
        \int_{t_0}^{\tau_1} \bm{f}(\bm{y}(s), \bm{z}(s), s)\,\mathrm{ds} - \sum_{j=1}^M q_{1,j} \bm{f}(\bm{y}(\tau_j), \bm{z}(\tau_j), \tau_j) = \mathcal{O}(\Delta t^{M + 1}).
    \end{equation*}
    The second error can be estimated as
    \begin{equation*}
        \begin{split}
            \sum_{j=1}^M q_{1,j} \bm{f}(\bm{y}(\tau_j), \bm{z}(\tau_j), \tau_j) &- \sum_{j=1}^M q_{1,j} \bm{f}(\bm{y}^0_j, \bm{z}^0_j, \tau_j) = \mathcal{O}(\Delta t^2)
        \end{split}
    \end{equation*}
    via the second-order Taylor series expansion of $\bm{f}(\bm{y}(\tau_j), \bm{z}(\tau_j), \tau_j)$ around the point $(y(t), z(t), t)=(\bm{y}^0_j, \bm{z}^0_j, \tau_j)$ by
    \begin{equation*}
        \begin{split}
            \bm{f}(\bm{y}(\tau_j), \bm{z}(\tau_j), \tau_j) 
            &= \bm{f}(\bm{y}^0_j, \bm{z}^0_j, \tau_j) + \bm{J}_{\bm{f}}^{\bm{y}} (\bm{y}^0_j, \bm{z}^0_j, \tau_j) \bm{e}^0_{\bm{y},j} + \bm{J}_{\bm{f}}^{\bm{z}} (\bm{y}^0_j, \bm{z}^0_j, \tau_j) \bm{e}^0_{\bm{z},j} \\
            &\quad+ \mathcal{O}((\bm{e}^0_{\bm{y},j})^2 + (\bm{e}^0_{\bm{z},j})^2),
        \end{split}
    \end{equation*}
    \prettyref{lm:quad}, \prettyref{lm:bounded_J_f}, and \prettyref{thm:lte_initial_SDC_C} for $j=1,..,M$.
    Thus, we get
    \begin{equation}
        \int_{t_0}^{\tau_1} \bm{f}(\bm{y}(s), \bm{z}(s), s)\,\mathrm{ds} - \sum_{j=1}^M q_{1,j} \bm{f}(\bm{y}^0_j, \bm{z}^0_j, \tau_j) = \mathcal{O}\left(\Delta t^{\min(2, M + 1)}\right). \label{eq:quadrature_error}
    \end{equation}
    In total, the error in \prettyref{eq:base_eq} modifies to 
    \begin{equation}
        \bm{e}^1_{\bm{y},1} = \mathcal{O}(\Delta t^{\min(2, M + 1)}) - \tilde{q}_{1,1} (\bm{f}(\bm{y}^1_1, \bm{z}^1_1, \tau_1) - \bm{f}(\bm{y}^0_1, \bm{z}^0_1, \tau_1)). \label{eq:base_eq_modified}
    \end{equation}
    We expand $\bm{f}(\bm{y}^1_1, \bm{z}^1_1, \tau_1)$ in a second-order Taylor series around $(y(t), z(t), t)=(\bm{y}^0_1, \bm{z}^0_1, \tau_1)$ by
    \begin{equation*}
        \begin{split}
            \bm{f}(\bm{y}^1_1, \bm{z}^1_1, \tau_1) &= \bm{f}(\bm{y}^0_1, \bm{z}^0_1, \tau_1) + \bm{J}_{\bm{f}}^{\bm{y}} (\bm{y}^0_1, \bm{z}^0_1, \tau_1) (\bm{y}^1_1 - \bm{y}^0_1) \\
            &\quad+ \bm{J}_{\bm{f}}^{\bm{z}} (\bm{y}^0_1, \bm{z}^0_1, \tau_1) (\bm{z}^1_1 - \bm{z}^0_1) + \mathcal{O}((\bm{y}^1_1 - \bm{y}^0_1)^2 + (\bm{z}^1_1 - \bm{z}^0_1)^2),
        \end{split}
    \end{equation*}
    and the error equation \prettyref{eq:base_eq_modified} becomes
    \begin{equation}
        \begin{split}
            \bm{e}^1_{\bm{y},1} &= \mathcal{O}(\Delta t^{\min(2, M + 1)}) - \tilde{q}_{1,1} \bm{J}_{\bm{f}}^{\bm{y}}(\bm{y}^0_1, \bm{z}^0_1, \tau_1) (\bm{y}^1_1 - \bm{y}^0_1) \\
            &\quad- \tilde{q}_{1,1} \bm{J}_{\bm{f}}^{\bm{z}}(\bm{y}^0_1, \bm{z}^0_1, \tau_1) (\bm{z}^1_1 - \bm{z}^0_1) + \mathcal{O}((\bm{y}^1_1 - \bm{y}^0_1)^2 + (\bm{z}^1_1 - \bm{z}^0_1)^2)).
        \end{split}\label{eq:error_Taylor}
    \end{equation}
    Manipulating the terms $\bm{y}^1_1 - \bm{y}^0_1$ and $\bm{z}^1_1 - \bm{z}^0_1$ by adding zeros, 
    we obtain
    \begin{equation}
        \begin{split}
            \bm{e}^1_{\bm{y},1} 
            &= \mathcal{O}(\Delta t^{\min(2, M + 1)}) + \tilde{q}_{1,1} \bm{J}_{\bm{f}}^{\bm{y}}(\bm{y}^0_1, \bm{z}^0_1, \tau_1) \bm{e}^1_{\bm{y},1} + \tilde{q}_{1,1} \bm{J}_{\bm{f}}^{\bm{z}}(\bm{y}^0_1, \bm{z}^0_1, \tau_1) \bm{e}^1_{\bm{z},1} \\
            &\quad+ \mathcal{O}((\bm{e}^1_{\bm{y},1})^2 + (\bm{e}^1_{\bm{z},1})^2)
        \end{split} \label{eq:error_Taylor_simplify}
    \end{equation}
    via \prettyref{lm:Q_IE_EE}, \prettyref{lm:bounded_J_f}, and \prettyref{thm:lte_initial_SDC_C}.
    Substituting $\bm{e}^1_{\bm{z},1}$ from \prettyref{eq:LTE_z}, the obtained error equation becomes independent of the error in $\bm{z}$. It is given by
    \begin{equation*}
        \begin{split}
            \bm{e}^1_{\bm{y},1} &= \mathcal{O}(\Delta t^{\min(2, M + 1)}) + \tilde{q}_{1,1} \bm{J}_{\bm{f}}^{\bm{y}} (\bm{y}^0_1, \bm{z}^0_1, \tau_1) \bm{e}^1_{\bm{y},1} \\
            &\quad+ \tilde{q}_{1,1} \bm{J}_{\bm{f}}^{\bm{z}} (\bm{y}^0_1, \bm{z}^0_1, \tau_1) \bm{J}_{\bm{G}} (\bm{y}(\tau_1), \tau_1) \bm{e}^1_{\bm{y},1} + \mathcal{O}((\bm{e}^1_{\bm{y},1})^2).
        \end{split}
    \end{equation*}
    Reformulating the equation yields the system
    \begin{equation*}
        \begin{split}
            &(\mathbf{I}_{n_d} - \tilde{q}_{1,1} \bm{J}_{\bm{f}}^{\bm{y}} (\bm{y}^0_1, \bm{z}^0_1, \tau_1) - \tilde{q}_{1,1} \bm{J}_{\bm{f}}^{\bm{z}}(\bm{y}^0_1, \bm{z}^0_1, \tau_1) \bm{J}_{\bm{G}} (\bm{y}(\tau_1), \tau_1)) \bm{e}^1_{\bm{y},1} \\
            &\qquad\qquad\qquad\qquad\qquad\qquad\qquad\qquad\qquad= \mathcal{O}(\Delta t^{\min(2, M + 1)}) + \mathcal{O}((\bm{e}^1_{\bm{y},1})^2).
        \end{split}
    \end{equation*}
    We assume that the operator on the left-hand side is invertible and the system can be solved for $\bm{e}^1_{\bm{y},1}$. Using \prettyref{lm:bounded_J_G} and \prettyref{lm:bounded_J_f}, the Jacobians $\bm{J}_{\bm{f}}^{\bm{y}}$, $\bm{J}_{\bm{f}}^{\bm{z}}$ and $\bm{J}_{\bm{G}}$ are bounded and can be estimated as $\mathcal{O}(1)$. Therefore, the inverse of the left-hand side operator can also be estimated as $\mathcal{O}(1)$. We obtain with \prettyref{lm:Q_IE_EE}
    \begin{align}
        \bm{e}^1_{\bm{y},1} = \mathcal{O}(\Delta t^{\min(2, M + 1)}),
    \end{align}
    and thus $\bm{e}^1_{\bm{z},1} = \mathcal{O}(\Delta t^{\min(2, M + 1)})$.

    In the induction hypothesis, we assume for a particular $k < 2M - 1$ that the error estimations $\bm{e}^k_{\bm{y}, 1} = \mathcal{O}(\Delta t^{k + 1})$ and $\bm{e}^k_{\bm{z}, 1} = \mathcal{O}(\Delta t^{k + 1})$ hold. Using \prettyref{eq:ift}, the local error in $\bm{z}$ is expressed as
    \begin{equation}
        \label{eq:LTE_z_ind_step}
        \begin{split}
            \bm{e}^{k + 1}_{\bm{z},1} &= \bm{G}(\bm{y}(\tau_1), \tau_1) - \bm{G}(\bm{y}^{k + 1}_1, \tau_1) = \bm{J}_{\bm{G}} (\bm{y}(\tau_1), \tau_1) \bm{e}^{k + 1}_{\bm{y},1} + \mathcal{O}((\bm{e}^{k + 1}_{\bm{y},1})^2),
        \end{split}
    \end{equation}
    again by expanding $\bm{G}(\bm{y}^{k + 1}_1, \tau_1)$ in a Taylor series around $(\bm{y} (t), t) = (\bm{y}(\tau_1), \tau_1)$ as in \prettyref{eq:LTE_z}. Taking the difference of the evaluated Picard integral equation \prettyref{eq:constrained_picard_integral} at $t = \tau_1$ and the ($k+1$)-th iterate obtained by equation \prettyref{eq:sdc_c_eq1}, the local error in $\bm{y}$ has the form
    \begin{equation}
        \begin{split}
            \bm{e}^{k + 1}_{\bm{y},1} 
            &= \int_{t_0}^{\tau_1} \bm{f}(\bm{y}(s), \bm{z}(s), s)\,\mathrm{ds} - \sum_{j=1}^M q_{1,j} \bm{f}(\bm{y}^k_j, \bm{z}^k_j, \tau_j) \\
            &\quad- \tilde{q}_{1,1} (\bm{f}(\bm{y}^{k + 1}_1, \bm{z}^{k + 1}_1, \tau_1) - \bm{f}(\bm{y}^k_1, \bm{z}^k_1, \tau_1)). \label{eq:base_eq2}
        \end{split}
    \end{equation}
    With the same steps as above and using the induction hypothesis, we obtain
    \begin{equation*}
        \int_{t_0}^{\tau_1} \bm{f}(\bm{y}(s), \bm{z}(s), s)\,\mathrm{ds} - \sum_{j=1}^M q_{1,j} \bm{f}(\bm{y}^k_j, \bm{z}^k_j, \tau_j) = \mathcal{O}(\Delta t^{\min(k + 2, M + 1)})
    \end{equation*}
    and the error equation \prettyref{eq:base_eq2} becomes
    \begin{equation*}
        \bm{e}^{k + 1}_{\bm{y},1} = \mathcal{O}(\Delta t^{\min(k + 2, M + 1)}) - \tilde{q}_{1,1} (\bm{f}(\bm{y}^{k + 1}_1, \bm{z}^{k + 1}_1, \tau_1) - \bm{f}(\bm{y}^k_1, \bm{z}^k_1, \tau_1)).
    \end{equation*}
    Expanding $\bm{f}(\bm{y}^{k+1}_1, \bm{z}^{k+1}_1, \tau_1)$ in a second-order Taylor series around the point $(y(t), z(t), t)=(\bm{y}^k_1, \bm{z}^k_1, \tau_1)$, the error modifies to
    \begin{equation}
        \begin{split}
            \bm{e}^{k + 1}_{\bm{y},1} &= \mathcal{O}(\Delta t^{\min(k + 2, M + 1)}) - \tilde{q}_{1,1} \bm{J}_{\bm{f}}^{\bm{y}} (\bm{y}^k_1, \bm{z}^k_1, \tau_1) (\bm{y}^{k + 1}_1 - \bm{y}^k_1) \\
            &\quad- \tilde{q}_{1,1} \bm{J}_{\bm{f}}^{\bm{z}} (\bm{y}^k_1, \bm{z}^k_1, \tau_1) (\bm{z}^{k + 1}_1 - \bm{z}^k_1) + \mathcal{O}((\bm{y}^{k + 1}_1 - \bm{y}^k_1)^2 + (\bm{z}^{k + 1}_1 - \bm{z}^k_1)^2)).
        \end{split} \label{eq:error_Taylor_ind_step}
    \end{equation}
    Adding zero terms in $\bm{y}^{k + 1}_1 - \bm{y}^k_1$ and $\bm{z}^{k + 1}_1 - \bm{z}^k_1$ and using the induction hypothesis, we simply get for \prettyref{eq:error_Taylor_ind_step}
    \begin{equation*}
        \begin{split}
            \bm{e}^{k + 1}_{\bm{y},1} &= \mathcal{O}(\Delta t^{\min(k + 2, M + 1)}) + \tilde{q}_{1,1} \bm{J}_{\bm{f}}^{\bm{y}} (\bm{y}^k_1, \bm{z}^k_1, \tau_1) \bm{e}^{k + 1}_{\bm{y},1} + \tilde{q}_{1,1} \bm{J}_{\bm{f}}^{\bm{z}} (\bm{y}^k_1, \bm{z}^k_1, \tau_1) \bm{e}^{k + 1}_{\bm{z},1}\\
            &\quad+ \mathcal{O}((\bm{e}^{k + 1}_{\bm{y},1})^2 + (\bm{e}^{k + 1}_{\bm{z},1})^2)
        \end{split}
    \end{equation*}
    similar to \prettyref{eq:error_Taylor_simplify}.
    Substituting $\bm{e}^{k + 1}_{\bm{z},1}$ from \prettyref{eq:LTE_z_ind_step} and reformulating, we obtain a system for $\bm{e}^{k + 1}_{\bm{y},1}$ given by
    \begin{equation}
        \begin{split}
            &(\mathbf{I}_{n_d} - \tilde{q}_{1,1} \bm{J}_{\bm{f}}^{\bm{y}} (\bm{y}^k_1, \bm{z}^k_1, \tau_1) - \tilde{q}_{1,1} \bm{J}_{\bm{f}}^{\bm{z}} (\bm{y}^k_1, \bm{z}^k_1, \tau_1) \bm{J}_{\bm{G}} (\bm{y}(\tau_1), \tau_1)) \bm{e}^{k + 1}_{\bm{y},1} \\
            &\qquad\qquad\qquad\qquad\qquad\qquad= \mathcal{O}(\Delta t^{\min(k + 2, M + 1)}) + \mathcal{O}((\bm{e}^{k + 1}_{\bm{y},1})^2).
        \end{split} \label{eq:linear_system_err_y}
    \end{equation}
    We assume that the operator on the left-hand side of the system \prettyref{eq:linear_system_err_y} is invertible. Using \prettyref{lm:bounded_J_G} and \prettyref{lm:bounded_J_f}, the Jacobians $\bm{J}_{\bm{f}}^{\bm{y}}$, $\bm{J}_{\bm{f}}^{\bm{z}}$, and $\bm{J}_{\bm{G}}$ are bounded and can be estimated as $\mathcal{O}(1)$. The inverse of the operator is therefore estimated as $\mathcal{O}(1)$. Using \prettyref{lm:Q_IE_EE}, we obtain
    \begin{align*}
        \bm{e}^{k + 1}_{\bm{y},1} = \mathcal{O}(\Delta t^{\min(k + 2, M + 1)}) \quad \text{and} \quad \bm{e}^{k + 1}_{\bm{z},1} = \mathcal{O}(\Delta t^{\min(k + 2, M + 1)}).
    \end{align*}
\end{proof}

\begin{theorem}
    \label{thm:lte_sdc_c}
    Under the assumptions of \prettyref{thm:lte_initial_SDC_C} the numerical solutions in iteration $k$ at arbitrary node $\tau_m$ obtained by the \texttt{SDC-C} scheme satisfy
    \begin{align*}
        \bm{e}^k_{\bm{y},m} = \mathcal{O}(\Delta t^{k + 1}) \quad \text{and} \quad \bm{e}^k_{\bm{z},m} = \mathcal{O}(\Delta t^{k + 1}).
    \end{align*}
    for $m = 1, \dots ,M$ and $k=0,1,..,2M-1$.
\end{theorem}

\begin{proof}[by induction]
    \prettyref{thm:lte_initial_SDC_C} and \prettyref{thm:lte_iters_first} provide the base case. In the induction hypothesis, we assume that
    \begin{align*}
        \bm{e}^k_{\bm{y},m} &= \mathcal{O}(\Delta t^{k + 1}),\quad \bm{e}^k_{\bm{z},m} = \mathcal{O}(\Delta t^{k + 1})\quad \text{for all }m, \tag{IH1} \label{eq:IH1}\\
        \bm{e}^{k + 1}_{\bm{y},\ell} &= \mathcal{O}(\Delta t^{k + 2}),\quad \bm{e}^{k + 1}_{\bm{z},\ell} = \mathcal{O}(\Delta t^{k + 2})\quad \text{with }\ell < m, \tag{IH2} \label{eq:IH2}
    \end{align*}
    for $m = 1,..,M$ and $k \le 2M - 1$.

    We consider the errors in iteration $k + 1$. As in \prettyref{eq:LTE_z}, the local error in $\bm{z}$ is given by
    \begin{equation}
        \begin{split}
            \bm{e}^{k + 1}_{\bm{z},m} 
            &= \bm{J}_{\bm{G}} (\bm{y}(\tau_m), \tau_m) \bm{e}^{k + 1}_{\bm{y},m} + \mathcal{O}((\bm{e}^{k + 1}_{\bm{y},m})^2),
        \end{split} \label{eq:LTE_z_2}
    \end{equation}
    via a Taylor series expansion of $\bm{G}(\bm{y}^{k + 1}_m, \tau_m)$ around $(\bm{y} (t), t) = (\bm{y}(\tau_m), \tau_m)$
    \begin{equation*}
        \begin{split}
            \bm{G}(\bm{y}^{k + 1}_m, \tau_m) 
            &= \bm{G}(\bm{y}(\tau_m), \tau_m) - \bm{J}_{\bm{G}} (\bm{y}(\tau_m), \tau_m) \bm{e}^{k + 1}_{\bm{y},m} + \mathcal{O}((\bm{e}^{k + 1}_{\bm{y},m})^2).
        \end{split}
    \end{equation*}
    Taking the difference of the evaluated Picard integral equation \prettyref{eq:constrained_picard_integral} at $t = \tau_m$ and the ($k+1$)-th iterate obtained by equation \prettyref{eq:sdc_c_eq1}, the local error in $\bm{y}$ has the form
    \begin{equation}
        \begin{split}
            \bm{e}^{k+1}_{\bm{y},m} &= \int_{t_0}^{\tau_m} \bm{f}(\bm{y}(s), \bm{z}(s), s)\,\mathrm{ds} - \sum_{j=1}^M q_{m,j} \bm{f}(\bm{y}^k_j, \bm{z}^k_j, \tau_j) \\
            &\quad- \sum_{j=1}^m \tilde{q}_{m,j} (\bm{f}(\bm{y}^{k + 1}_j, \bm{z}^{k + 1}_j, \tau_j) - \bm{f}(\bm{y}^k_j, \bm{z}^k_j, \tau_j)).
        \end{split} \label{eq:base_eq_2}
    \end{equation}
    The spectral quadrature error can be estimated as in the proof for \prettyref{thm:lte_iters_first}: Adding a zero term to it yields
    \begin{subequations} \label{eq:first_quantity_2}
        \begin{align}
            &\int_{t_0}^{\tau_m} \bm{f}(\bm{y}(s), \bm{z}(s), s)\,\mathrm{ds} - \sum_{j=1}^M q_{m,j} \bm{f}(\bm{y}^k_j, \bm{z}^k_j, \tau_j) \label{eq:first_quantity_2_eq1} \\
            &\qquad\qquad= \left(\int_{t_0}^{\tau_m} \bm{f}(\bm{y}(s), \bm{z}(s), s)\,\mathrm{ds} - \sum_{j=1}^M q_{m,j} \bm{f}(\bm{y}(\tau_j), \bm{z}(\tau_j), \tau_j)\right) \label{eq:first_quantity_2_eq2} \\
            &\qquad\qquad\quad+ \left(\sum_{j=1}^M q_{m,j} \bm{f}(\bm{y}(\tau_j), \bm{z}(\tau_j), \tau_j) - \sum_{j=1}^M q_{m,j} \bm{f}(\bm{y}^k_j, \bm{z}^k_j, \tau_j)\right). \label{eq:first_quantity_2_eq3}
        \end{align}
    \end{subequations}
    The quadrature rule using weights $q_{m, j}$ is interpreted as a numerical approximation of the integral of an interpolating polynomial through the points $(\bm{y}(\tau_j), \bm{z}(\tau_j), \tau_j)$ because the rule only evaluates $\bm{f}$ at the collocation nodes $\tau_j$ for $j=1,..,M$. The integral is therefore computed with error $\mathcal{O}(\Delta t^M)$, i.e.,
    \begin{equation}
        \int_{t_0}^{\tau_m} \bm{f}(\bm{y}(s), \bm{z}(s), s)\,\mathrm{ds} - \sum_{j=1}^M q_{m,j} \bm{f}(\bm{y}(\tau_j), \bm{z}(\tau_j), \tau_j) = \mathcal{O}(\Delta t^{M + 1}). \label{eq:high_order_integral_approximation}
    \end{equation}
    The second error in \prettyref{eq:first_quantity_2_eq3} is estimated as
    \begin{equation*}
        \begin{split}
            \sum_{j=1}^M q_{m,j} \bm{f}(\bm{y}(\tau_j), \bm{z}(\tau_j), \tau_j) &- \sum_{j=1}^M q_{m,j} \bm{f}(\bm{y}^k_j, \bm{z}^k_j, \tau_j) = \mathcal{O}(\Delta t^{k + 2})
        \end{split}
    \end{equation*}
    by expanding $\bm{f}(\bm{y}(\tau_j), \bm{z}(\tau_j), \tau_j)$ in the second-order Taylor series expansion around $(y(t),z(t),t) = (\bm{y}^k_j, \bm{z}^k_j, \tau_j)$
    \begin{equation*}
        \begin{split}
            \bm{f}(\bm{y}(\tau_j), \bm{z}(\tau_j), \tau_j) 
            &= \bm{f}(\bm{y}^k_j, \bm{z}^k_j, \tau_j) + \bm{J}_{\bm{f}}^{\bm{y}} (\bm{y}^k_j, \bm{z}^k_j, \tau_j) \bm{e}^k_{y,j} + \bm{J}_{\bm{f}}^{\bm{z}} (\bm{y}^k_j, \bm{z}^k_j, \tau_j) \bm{e}^k_{\bm{z},j} \\
            &\quad+ \mathcal{O}((\bm{e}^k_{\bm{y},j})^2 + (\bm{e}^k_{\bm{z},j})^2),
        \end{split}
    \end{equation*}
    \prettyref{lm:quad}, \prettyref{lm:bounded_J_f}, and \prettyref{thm:lte_initial_SDC_C} for $j=1,..,M$.
    Therefore, the spectral quadrature error can be estimated as
    \begin{equation}
        \int_{t_0}^{\tau_m} \bm{f}(\bm{y}(s), \bm{z}(s), s)\,\mathrm{ds} - \sum_{j=1}^M q_{m,j} \bm{f}(\bm{y}^k_j, \bm{z}^k_j, \tau_j) = \mathcal{O}\left(\Delta t^{\min(k + 2, M + 1)}\right), \label{eq:quadrature_error_2}
    \end{equation}
    and the error equation becomes
    \begin{equation*}
        \bm{e}^{k + 1}_{\bm{y},m} = \mathcal{O}(\Delta t^{\min(k + 2, M + 1)}) - \sum_{j=1}^m \tilde{q}_{m,j} (\bm{f}(\bm{y}^{k + 1}_j, \bm{z}^{k + 1}_j, \tau_j) - \bm{f}(\bm{y}^k_j, \bm{z}^k_j, \tau_j)).
    \end{equation*}

    For $j = 1,..,m$, we expand $\bm{f}(\bm{y}^{k + 1}_j, \bm{z}^{k + 1}_j, \tau_j)$ in a second-order Taylor series around $(y(t),z(t),t)=(\bm{y}^k_j, \bm{z}^k_j, \tau_j)$ by
    \begin{equation*}
        \begin{split}
            \bm{f}(\bm{y}^{k + 1}_j, \bm{z}^{k + 1}_j, \tau_j) &= \bm{f}(\bm{y}^k_j, \bm{z}^k_j, \tau_j) + \bm{J}_{\bm{f}}^{\bm{y}} (\bm{y}^k_j, \bm{z}^k_j, \tau_j) (\bm{y}^{k + 1}_j - \bm{y}^{k + 1}_j) \\
            &\quad+ \bm{J}_{\bm{f}}^{\bm{z}} (\bm{y}^k_j, \bm{z}^k_j, \tau_j) (\bm{z}^{k + 1}_j - \bm{z}^k_j) \\
            &\quad+ \mathcal{O}((\bm{y}^{k + 1}_j - \bm{y}^k_j)^2 + (\bm{z}^{k + 1}_j - \bm{z}^k_j)^2),
        \end{split}
    \end{equation*}
    and the error modifies to
    \begin{equation}
        \begin{split}
            \bm{e}^{k + 1}_{\bm{y},m} &= \mathcal{O}(\Delta t^{\min(k + 2, M + 1)}) \\
            &\quad- \sum_{j=1}^m \tilde{q}_{m,j} (\bm{J}_{\bm{f}}^{\bm{y}} (\bm{y}^k_j, \bm{z}^k_j, \tau_j) (\bm{y}^{k + 1}_j - \bm{y}^k_j) +  \bm{J}_{\bm{f}}^{\bm{z}} (\bm{y}^k_j, \bm{z}^k_j, \tau_j) (\bm{z}^{k + 1}_j - \bm{z}^k_j)) \\
            &\quad+ \mathcal{O}((\bm{y}^{k + 1}_j - \bm{y}^k_j)^2 + (\bm{z}^{k + 1}_j - \bm{z}^k_j)^2)).
        \end{split} \label{eq:error_Taylor_thm27}
    \end{equation}
    Adding zero terms to $\bm{y}^{k + 1}_j - \bm{y}^k_j$ and $\bm{z}^{k + 1}_j - \bm{z}^k_j$, and using the induction hypothesis, we conclude that
    \begin{equation} \label{eq:zero_terms_2}
        \begin{split}
            \bm{y}^{k + 1}_j - \bm{y}^k_j = \bm{y}^{k + 1}_j - \bm{y}(\tau_j) + \bm{y}(\tau_j) - \bm{y}^k_j = -\bm{e}^{k + 1}_{\bm{y},j} + \mathcal{O}(\Delta t^{k + 1}), \\
            \bm{z}^{k + 1}_j - \bm{z}^k_j = \bm{z}^{k + 1}_j - \bm{z}(\tau_j) + \bm{z}(\tau_j) - \bm{z}^k_j = -\bm{e}^{k + 1}_{\bm{z},j} + \mathcal{O}(\Delta t^{k + 1}).
        \end{split}
    \end{equation}
    Using \prettyref{eq:zero_terms_2} in \prettyref{eq:error_Taylor_thm27}, we obtain
    \begin{equation}
        \begin{split}
            \bm{e}^{k + 1}_{\bm{y},m} 
            &= \mathcal{O}(\Delta t^{\min(k + 2, M + 1)}) \\
            &\quad+ \sum_{j=1}^m \tilde{q}_{m,j} (\bm{J}_{\bm{f}}^{\bm{y}} (\bm{y}^k_j, \bm{z}^k_j, \tau_j) \bm{e}^{k + 1}_{\bm{y},j} + \bm{J}_{\bm{f}}^{\bm{z}} (\bm{y}^k_j, \bm{z}^k_j, \tau_j) \bm{e}^{k + 1}_{\bm{z},j}) \\
            &\quad+ \sum_{j=1}^m \mathcal{O}((\bm{e}^{k + 1}_{\bm{y},j})^2 + (\bm{e}^{k + 1}_{\bm{z},j})^2)
        \end{split} \label{eq:err_eq_thm3}
    \end{equation}
    via \prettyref{lm:Q_IE_EE}, \prettyref{lm:bounded_J_f}, and \prettyref{thm:lte_initial_SDC_C}.
    Substituting $\bm{e}^{k + 1}_{\bm{z},j}$ from \prettyref{eq:LTE_z_2} (for $m = j$) into \prettyref{eq:err_eq_thm3}, the obtained local error in $\bm{y}$ becomes independent of the local error in $\bm{z}$. It is given by
    \begin{equation*}
        \begin{split}
            \bm{e}^{k + 1}_{\bm{y},m} &= \mathcal{O}(\Delta t^{\min(k + 2, M + 1)}) + \sum_{j=1}^m \tilde{q}_{m,j} \bm{J}_{\bm{f}}^{\bm{y}} (\bm{y}^k_j, \bm{z}^k_j, \tau_j) \bm{e}^{k + 1}_{\bm{y},j} \\
            &\quad+ \sum_{j=1}^m \tilde{q}_{m,j} \bm{J}_{\bm{f}}^{\bm{z}} (\bm{y}^k_j, \bm{z}^k_j, \tau_j) \bm{J}_{\bm{G}} (\bm{y}(\tau_j), \tau_j) \bm{e}^{k + 1}_{\bm{y},j} + \sum_{j=1}^m \mathcal{O}((\bm{e}^{k + 1}_{\bm{y},j})^2).
        \end{split}
    \end{equation*}
    Arranging the first $m-1$ terms of the sums gives
    \begin{equation*}
        \begin{split}
            \bm{e}^{k + 1}_{\bm{y},m} &= \mathcal{O}(\Delta t^{\min(k + 2, M + 1)}) \\
            &\quad + \sum_{j=1}^{m - 1} \tilde{q}_{m,j} (\bm{J}_{\bm{f}}^{\bm{y}} (\bm{y}^k_j, \bm{z}^k_j, \tau_j) + \bm{J}_{\bm{f}}^{\bm{z}} (\bm{y}^k_j, \bm{z}^k_j, \tau_j) \bm{J}_{\bm{G}} (\bm{y}(\tau_j), \tau_j)) \bm{e}^{k + 1}_{\bm{y},j} \\
            &\quad + \tilde{q}_{m,m} (\bm{J}_{\bm{f}}^{\bm{y}} (\bm{y}^k_m, \bm{z}^k_m, \tau_m) + \bm{J}_{\bm{f}}^{\bm{z}} (\bm{y}^k_m, \bm{z}^k_m, \tau_m) \bm{J}_{\bm{G}} (\bm{y}(\tau_m), \tau_m)) \bm{e}^{k + 1}_{\bm{y},m} \\
            &\quad+ \sum_{j=1}^{m-1} \mathcal{O}((\bm{e}^{k + 1}_{\bm{y},j})^2) + \mathcal{O}((\bm{e}^{k + 1}_{\bm{y},m})^2).
        \end{split}
    \end{equation*}
    The induction hypothesis \prettyref{eq:IH2} implies that the first $m - 1$ summands are even of order $\mathcal{O}(\Delta t^{k + 3})$, i.e.,
    \begin{equation*}
        \begin{split}
            &\sum_{j=1}^{m - 1} \tilde{q}_{m,j} (\bm{J}_{\bm{f}}^{\bm{y}} (\bm{y}^k_j, \bm{z}^k_j, \tau_j) + \bm{J}_{\bm{f}}^{\bm{z}} (\bm{y}^k_j, \bm{z}^k_j, \tau_j) \bm{J}_{\bm{G}} (\bm{y}(\tau_j), \tau_j)) \bm{e}^{k + 1}_{\bm{y},j}  = \mathcal{O}(\Delta t^{k + 3}),
        \end{split}
    \end{equation*}
    together with \prettyref{lm:Q_IE_EE}, \prettyref{lm:bounded_J_G} and \prettyref{lm:bounded_J_f}. Thus, we obtain
    \begin{equation*}
        \begin{split}
            \bm{e}^{k + 1}_{\bm{y},m} &= \tilde{q}_{m,m} \bm{J}_{\bm{f}}^{\bm{y}} (\bm{y}^k_m, \bm{z}^k_m, \tau_m) \bm{e}^{k + 1}_{\bm{y},m} + \tilde{q}_{m,m} \bm{J}_{\bm{f}}^{\bm{z}} (\bm{y}^k_m, \bm{z}^k_m, \tau_m) \bm{J}_{\bm{G}} (\bm{y}(\tau_m), \tau_m) \bm{e}^{k + 1}_{\bm{y},m} \\
            &\quad +\mathcal{O}(\Delta t^{\min(k + 2, M + 1)}) + \mathcal{O}((\bm{e}^{k + 1}_{\bm{y},m})^2)
        \end{split}
    \end{equation*}
    resulting in the system
    \begin{equation}
        \begin{split}
            &(\mathbf{I}_{n_d} - \tilde{q}_{m,m} \bm{J}_{\bm{f}}^{\bm{y}} (\bm{y}^k_m, \bm{z}^k_m, \tau_m) - \tilde{q}_{m,m} \bm{J}_{\bm{f}}^{\bm{z}} (\bm{y}^k_m, \bm{z}^k_m, \tau_m) \bm{J}_{\bm{G}} (\bm{y}(\tau_m), \tau_m)) \bm{e}^{k + 1}_{\bm{y},m} \\
            &\qquad\qquad\qquad\qquad= \mathcal{O}(\Delta t^{\min(k + 2, M + 1)}) + \mathcal{O}(( \bm{e}^{k + 1}_{\bm{y},m})^2).
        \end{split} \label{eq:o_system}
    \end{equation}
    Using \prettyref{lm:bounded_J_G} and \prettyref{lm:bounded_J_f}, the Jacobians $\bm{J}_{\bm{f}}^{\bm{y}}$, $\bm{J}_{\bm{f}}^{\bm{z}}$ and $\bm{J}_{\bm{G}}$ are bounded. Thus, the inverse of the operator on the left-hand side in \prettyref{eq:o_system} is estimated as $\mathcal{O}(1)$. Using \prettyref{lm:Q_IE_EE}, we deduce
    \begin{align*}
        \bm{e}^{k + 1}_{\bm{y},m} = \mathcal{O}(\Delta t^{\min(k + 2, M + 1)})
    \end{align*}
    and thus $\bm{e}^{k + 1}_{\bm{z},m} = \mathcal{O}(\Delta t^{\min(k + 2, M + 1)})$.
\end{proof}

\begin{figure}[!t]
    \centering
    \includegraphics[width=1.105\textwidth]{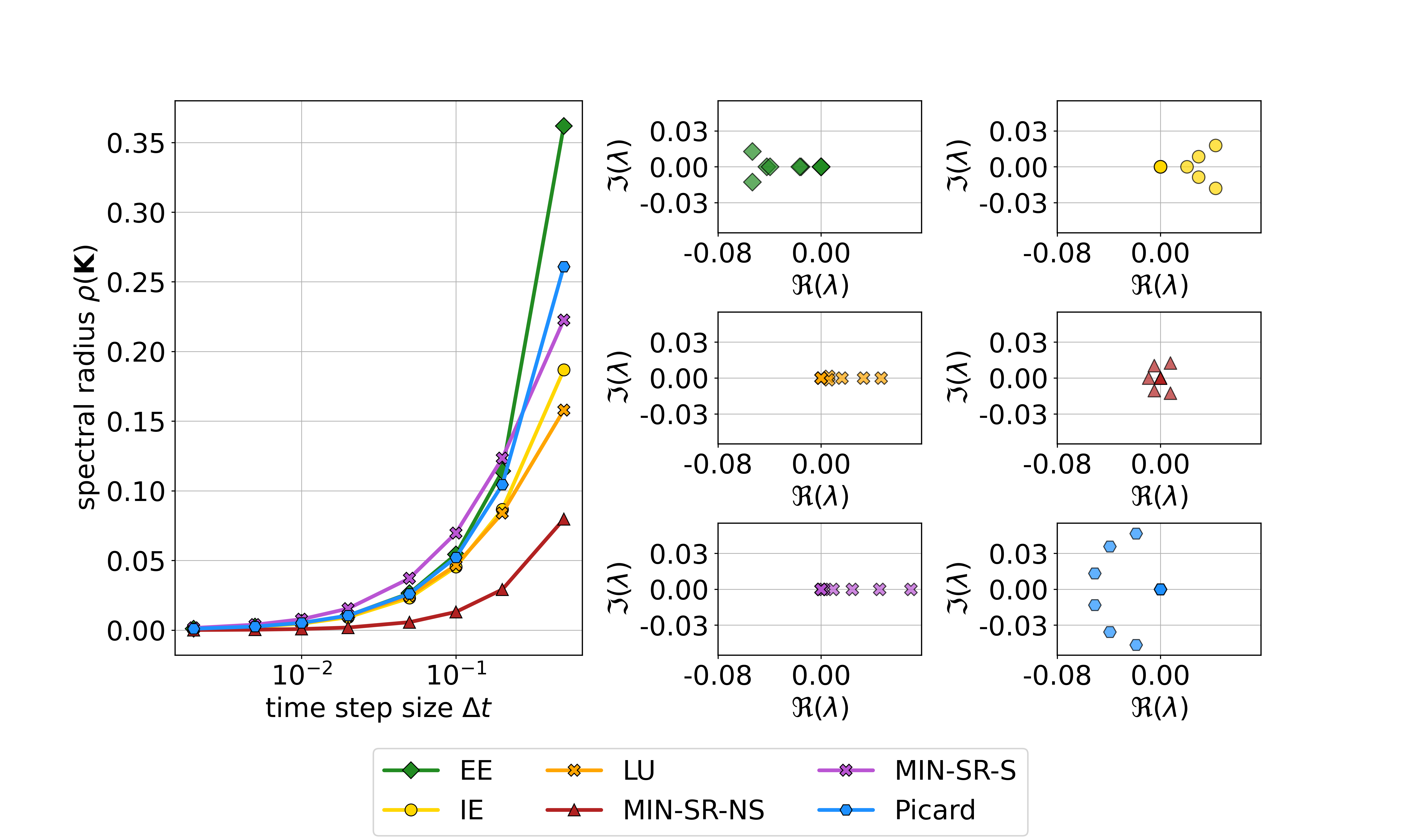}
    \caption{Spectral radius and eigenvalue distributions of iteration matrix of different \texttt{SDC-C} variants for the linear problem \prettyref{eq:linear_problem}. The schemes are based on $M = 6$ Radau IIA nodes. Left: Spectral radius against different time step sizes $\Delta t$. Right: Eigenvalue distributions (real part versus imaginary part) for $\Delta t = 0.01$.}
    \label{fig:fig2}
\end{figure}

\prettyref{lm:Q_IE_EE} states that the coefficients of $\bm{Q}_\Delta^{\texttt{IE}}$ and $\bm{Q}_\Delta^{\texttt{EE}}$ scale linearly with the time step size, but the result also carries over to the coefficients of $\bm{Q}_\Delta^{\texttt{MIN-SR-NS}}$ and $\bm{Q}_\Delta^{\texttt{Picard}}$. Therefore, the result of \prettyref{thm:lte_sdc_c} holds for \texttt{MIN-SR-NS} and \texttt{Picard} coefficients as well.

\begin{remark}
    The statement of \prettyref{thm:lte_sdc_c} can also be shown for \texttt{IMEX-SDC-C} \prettyref{eq:imex_sdc_c} with minor changes. In \prettyref{lm:bounded_J_f} and \prettyref{thm:lte_initial_SDC_C}, we require $\bm{f}_\mathrm{im}$ and $\bm{f}_\mathrm{ex}$ be $C^2$ functions and the proofs of \prettyref{lm:bounded_J_f} and \prettyref{thm:lte_initial_SDC_C} apply directly to the modified case. In order to show the statement of \prettyref{thm:lte_iters_first}, for \prettyref{eq:imex_constrained_picard_integral}, we can use the fact that $\bm{f} = \bm{f}_\mathrm{im} + \bm{f}_\mathrm{ex}$. Therefore, the high-order approximation of the integral in Picard's formulation is of order $\mathcal{O}(\Delta t^{M + 1})$ as in \prettyref{eq:high_order_integral_approximation}. For \texttt{IMEX-SDC-C}, \prettyref{thm:lte_iters_first} holds, but the explicit quadrature term corresponding to $q^\mathrm{ex}_{11}$ vanishes because the diagonal elements of the matrix $\bm{Q}_\Delta^\texttt{ex}$ are zero. It might be useful to show that the statement also holds for the second collocation node $\tau_2$, i.e.,
    \begin{equation}
        \bm{e}^k_{\bm{y},2} = \mathcal{O}(\Delta t^{k + 1}) \quad \text{and} \quad \bm{e}^k_{\bm{z},2} = \mathcal{O}(\Delta t^{k + 1})
    \end{equation}
    for $k = 0,1,..,2M - 1$, where the explicit quadrature term must be taken into account.
\end{remark}

The analysis of the index one case above can also be performed for DAEs of higher index. Here, some more assumptions must be taken into account \cite{Hairer1989}, \cite{Hairer_stiff2010}. For example, consider the semi-explicit DAE of index two of the specific form
\begin{equation}
    \bm{y}' (t) = \bm{f} (\bm{y}(t), \bm{z}(t), t), \qquad \qquad \bm{0} = \bm{g} (\bm{y}(t), t). \label{eq:semiexplicit_dae_index_two}
\end{equation}
Differentiating the algebraic constraints yields the hidden constraints
\begin{equation}
    \bm{0} = \bm{f}(\bm{y}(t), \bm{z}(t), t) \bm{J}_{\bm{g}}^{\bm{y}}(\bm{y}(t)),
\end{equation}
and the assumption of a bounded invertible Jacobian $\bm{J}^{\bm{z}}_{g}$ is insufficient. Instead, we need to assume that the Jacobian $\bm{J}^{\bm{y}}_{\bm{g}} \bm{J}^{\bm{z}}_{\bm{f}}$ is invertible and bounded. Stronger assumptions are also needed for problems of an index greater than two. An analysis of higher index problems is left for future.
    
\section{Numerical results} \label{sec:numerical_results}
In order to evaluate the performance and efficiency of \texttt{SDC-C} we examine the method using different choices of the $\bm{Q}_\Delta$ matrix introduced in \prettyref{sec:sdc_iter_method} in three test cases: A linear test DAE, the nonlinear problem describing Andrews' squeezing mechanism, and a reaction-diffusion problem. In particular, we investigate the scheme using diagonal matrices $\bm{Q}_\Delta^{\texttt{MIN-SR-NS}}$ and $\bm{Q}_\Delta^{\texttt{MIN-SR-S}}$, which are specifically designed to parallelize SDC across the method. Simulations have been performed with the Python implementation \texttt{pySDC} \cite{Speck2025} using \texttt{MPI} for distributed memory parallelism with \texttt{mpi4py=4.0.3} \cite{Dalcin2011}. The \texttt{SDC-C} method is compared to the variants \texttt{FI-SDC} and \texttt{SI-SDC}, see \prettyref{sec:sdc_dae}. All SDC variants are used with Radau IIA nodes. Instead of using the residual, the increment is monitored. We denote the solution as converged if the increment drops below an error tolerance $\etol$, i.e.,
\begin{equation}
    ||\bm{u}^{\tilde{k} + 1} - \bm{u}^{\tilde{k}}|| < \etol,
\end{equation}
where $\tilde{k}$ denotes the iteration number at which the solution is converged. For each problem, the error tolerance for each problem is set differently. The vector $\bm{u}^k_{M,t}$ defines the numerical solution of any unknown $\bm{u}$ after iteration $k$ at last collocation node $\tau_M$ at a time $t$. Note that $\tau_M = t$, the numerical solution is thus the one to the next time step. 


The comparison includes Radau IIA methods \texttt{RadauIIA5} and \texttt{RadauIIA7} of order $5$ and $7$, as they are common used \textit{methods} when solving DAEs. Radau IIA methods produce high-order numerical solutions, and they do not suffer from order reduction because they are stiffly accurate \cite{Hairer1989}. The class of half-explicit RK methods \cite{Hairer1989} is quite close to the \texttt{SDC-C} method. Thus, we also compare with a half-explicit RK method using the Dormand \& Prince formula \cite{Dormand1980} denoted by \texttt{DOPRI5} being a suitable method in this setting \cite{Hairer1989}. The computations were run on one compute node of the PLEIADES cluster at the University of Wuppertal, where we have used modules \texttt{GCC/12.3.0}, \texttt{Python/3.11.3}, and \texttt{OpenMPI/4.1.5}. More details about hardware and software are listed in \cite{PleiadesCluster}. For parallel solvers, $M$ \texttt{MPI} processes are used, where one process is assigned to one collocation node. The used number of nodes $M$ is specified in each case. The experiments are done for the entire time interval for different time step sizes if not explicitly mentioned. We study the different methods using the same time step sizes, i.e., they do not differ between the schemes. The study includes the accuracy of the numerical solution computed by the different methods, where we consider the $L_\infty$ error for the linear problem and the reaction-diffusion problem, that defines the errors on all variables and timepoints. For Andrews' squeezer, we consider the error in $\bm{q}$ at end time $T$. Computational costs are measured in terms of wall-clock times.

\begin{figure}[!t]
    \centering
    \includegraphics[width=\textwidth]{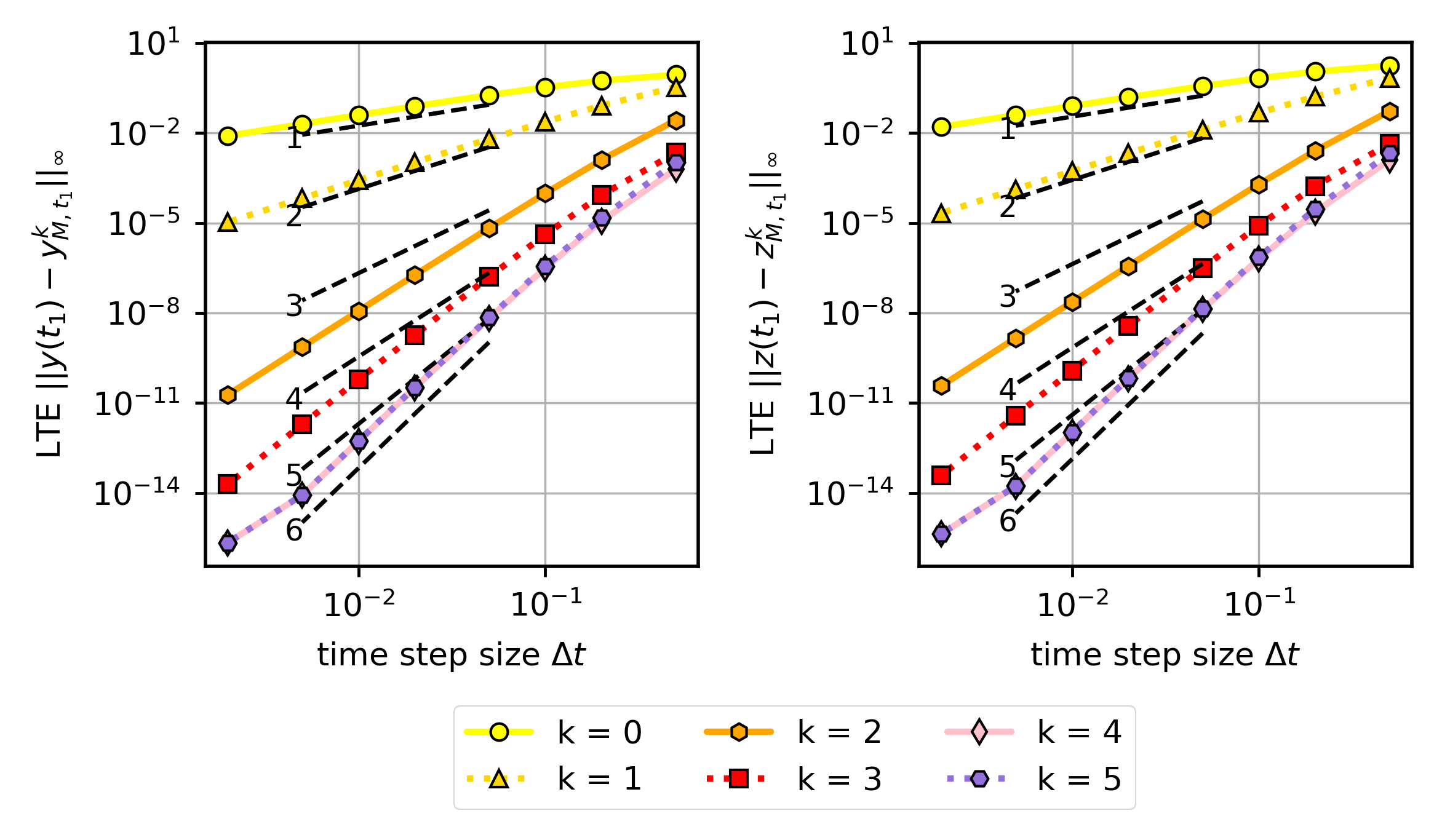}
    \caption{Local truncation error in $y$ and $z$ against time step sizes $\Delta t$ for each iteration $k = 0,\dots,2M - 1$. Orders of accuracy are shown for \texttt{SDC-C-MIN-SR-NS} based on $M = 3$ nodes for the linear problem \prettyref{eq:linear_problem}. Black dashed lines indicate the reference order. Left: Error in $y$. Right: Error in $z$.}
    \label{fig:fig3}
\end{figure}

\subsection{Linear problem} \label{sec:linear_problem}
First, we study the linear problem
\begin{equation}
    \begin{split}
        y' (t) &= -2 y(t) + z(t), \\
        0 &= -2 y(t) - z(t)
    \end{split} \label{eq:linear_problem}
\end{equation}
for scalar functions $y(t), z(t) \in \mathbb{R}$. The problem is studied for the time interval $[0, 1]$ with initial conditions $(y_0, z_0) = (1, -2)$, where
\begin{equation*}
        y(t) = e^{-4 t}, \qquad z(t) = -2 e^{-4 t}
\end{equation*}
are exact solutions of the linear problem. The linear implicit system at each node is solved directly. The increment tolerance is set to $\etol = 10^{-12}$.

\begin{figure}[!t]
    \centering
    \includegraphics[width=\textwidth]{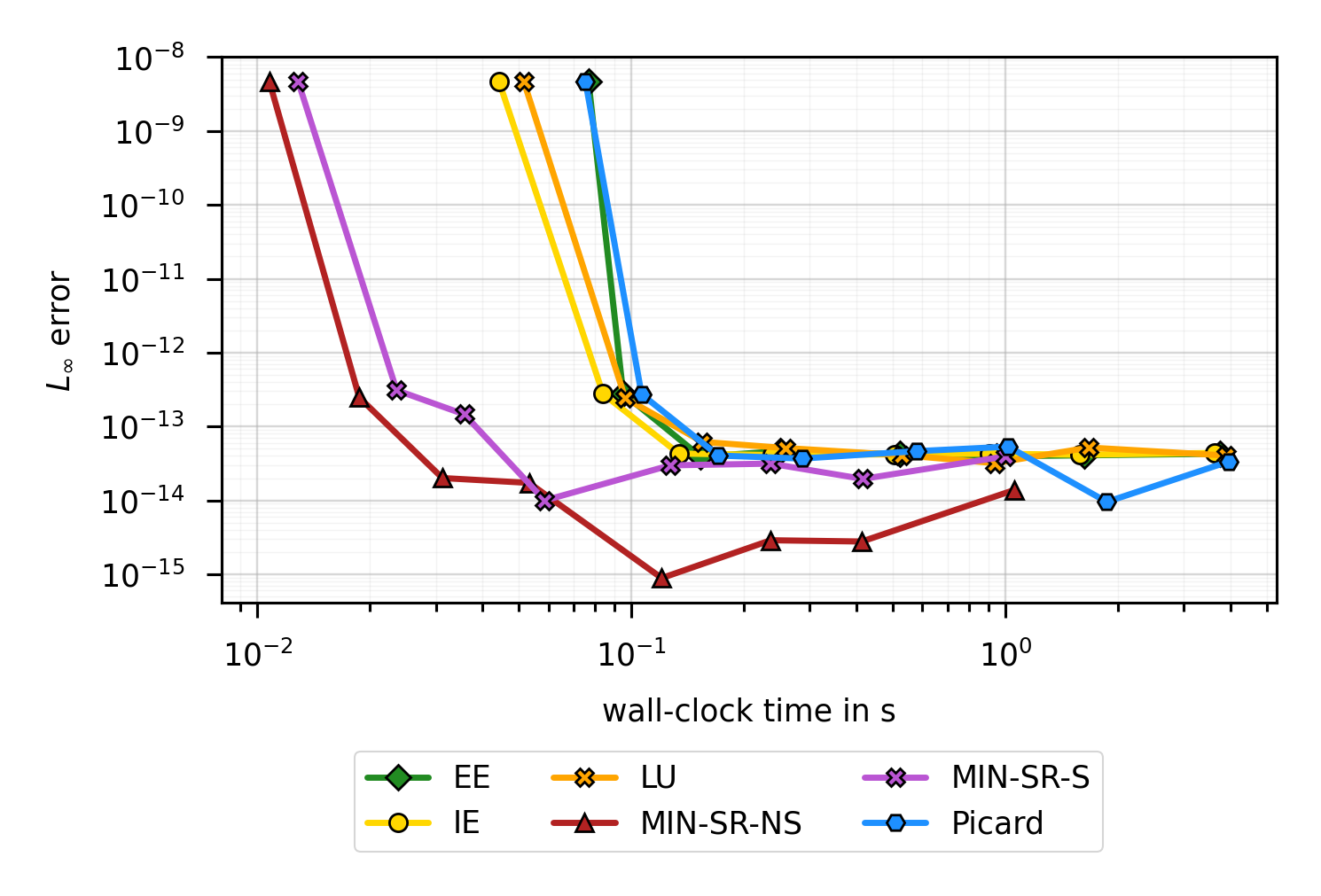}
    \caption{Wall-clock time against $L_{\infty}$ error of \texttt{SDC-C} methods for the linear problem \prettyref{eq:linear_problem}. The schemes are based on $M = 6$ Radau IIA nodes using several choices of $\bm{Q}_\Delta$. All methods used the same time step sizes.}
    \label{fig:fig4}
\end{figure}

For the linear problem, the \texttt{SDC-C} scheme can be written as a linear iterative method with a moderate normal iteration matrix as described in \prettyref{sec:sdc_iter_method}, where we define the iteration matrix $\bm{K}$ as moderate normal if the deviation
\begin{equation}
    d = \lVert \bm{K}\bm{K}^\top - \bm{K}^\top \bm{K}\rVert_\infty
\end{equation}
is not too far from zero. \prettyref{fig:fig2} shows the time step size against the spectral radius (the largest absolute eigenvalue) and the entire eigenvalue distribution of the \texttt{SDC-C} iteration matrix based on $M = 6$ nodes for different matrices $\bm{Q}_\Delta$. The \texttt{SDC-C-Picard} scheme exhibits a smaller spectral radius, implying faster convergence for $\Delta t = 0.01$ than the \texttt{MIN-SR-S} scheme. Studying the eigenvalue distributions for the \texttt{SDC-C-MIN-SR-S} scheme we found that there is only one large eigenvalue in magnitude in the distribution and all remaining staying closer to the origin. In contrast, the eigenvalues for the \texttt{Picard} iteration form a half circle around the origin with equal distance to it. Numerical experiments confirm that the numerical solution converges faster when using \texttt{SDC-C} with \texttt{MIN-SR-S} coefficients. All eigenvalues of the \texttt{MIN-SR-NS} scheme are close to the origin, indicating the fastest convergence compared to all other methods.

As mentioned in \prettyref{rem:convergence_alg_const} \texttt{FI-SDC} leads to errors in the algebraic constraints. Furthermore, the non-normality of \texttt{FI-SDC}'s iteration matrices renders the approach chosen for analysis here infeasible. Thus, the spectral radii and corresponding eigenvalue distributions do not provide any information about convergence.

The \texttt{SDC-C-MIN-SR-NS} method shows excellent convergence behavior because of the smallest spectral radius, and all eigenvalues of the corresponding iteration matrix are clustered around zero, see \prettyref{fig:fig2}. This is expected because the differential equation represents a non-stiff component in the semi-explicit linear DAE problem.

\begin{figure}[!t]
    \centering
    \includegraphics[width=\textwidth]{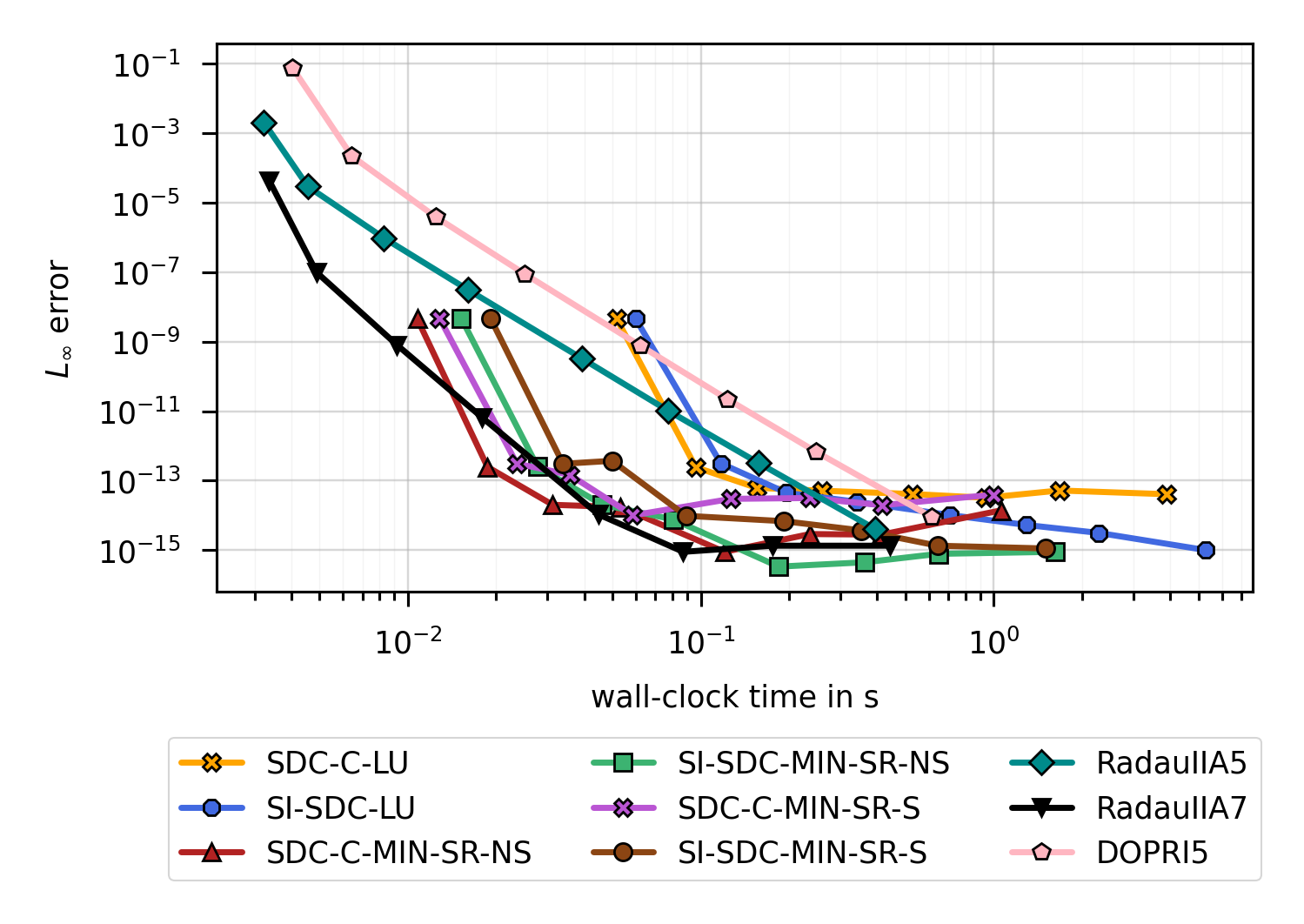}
    \caption{Wall-clock time versus $L_{\infty}$ error of different SDC variants and RK methods for the linear problem \prettyref{eq:linear_problem}. The SDC schemes are based on $M = 6$ using \texttt{LU}, \texttt{MIN-SR-NS}, and \texttt{MIN-SR-S} preconditioning. All methods used the same time step sizes.}
    \label{fig:fig5}
\end{figure}

\prettyref{fig:fig3} shows the time step sizes against the local truncation error (LTE) of the \texttt{SDC-C} method using \texttt{MIN-SR-NS} coefficients after each iteration $k$ as predicted by \prettyref{thm:lte_sdc_c}. After $k = 1$ iteration, the numerical solutions gain even two orders. Therefore, the maximum order is achieved after four iterations. Numerical experiments have shown that an order jump of two is achieved after $k = M - 1$ iterations, but the authors do not have a theoretical explanation. The observation was already made in \cite{Caklovic2025} and carries over to the DAE case studied here.

In \prettyref{fig:fig4}, the $L_\infty$ error against the wall-clock time of different \texttt{SDC-C} methods based on $M = 6$ nodes is shown. The parallel schemes are significantly faster than the serial \texttt{SDC-C} schemes. For time step sizes $\Delta t = 0.5, 0.2, 0.1, 0.05$, \texttt{SDC-C} using \texttt{MIN-SR-NS} coefficients is more efficient than using the \texttt{MIN-SR-S} coefficients, which is expected because the differential equation is non-stiff and algebraic constraints representing the stiff limit are treated implicitly without numerical integration. For the largest step size, the \texttt{IE} and \texttt{LU} methods perform better. In general, all sequential schemes have a similar runtime. For smaller $\Delta t$, the benefit in runtime allows the parallel methods to use larger time step sizes than the serial methods by maintaining high accuracy. However, this is dependent of the number of nodes $M$ because larger time step sizes result in less precision for smaller $M$. 

\begin{figure}[!t]
    \centering
    \includegraphics[width=\textwidth]{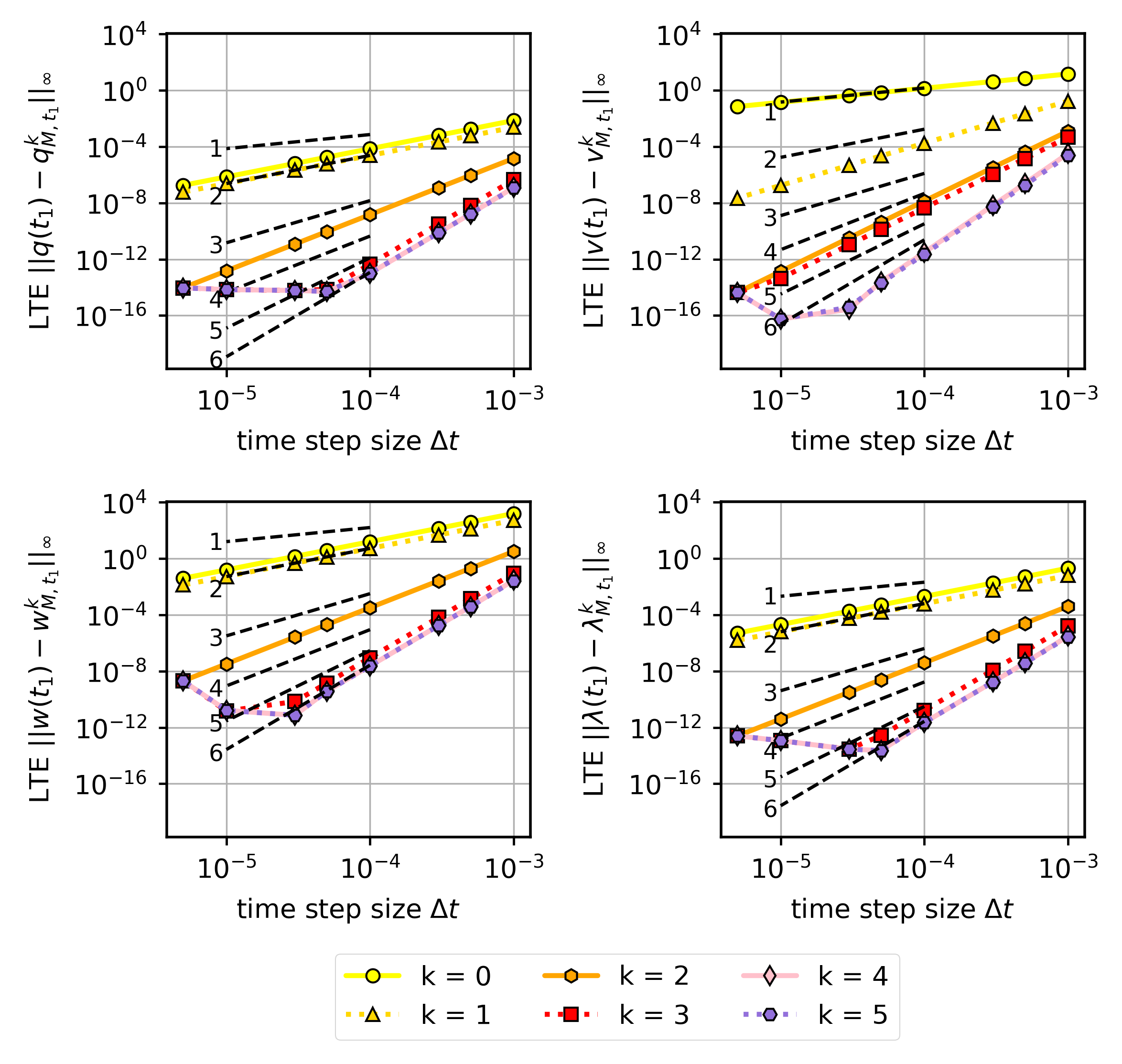}
    \caption{Local truncation error in $\bm{q}$, $\bm{v}$, $\bm{w}$, and $\bm{\lambda}$ against time step sizes $\Delta t$ for each iteration $k = 0,\dots,2M - 1$. Orders of accuracy are shown for \texttt{SDC-C-EE} based on $M = 3$ nodes for Andrews' squeezer \prettyref{eq:andrews_index_one}. Top row: Errors in $\bm{q}$ and $\bm{v}$. Bottom row: Errors in $\bm{w}$ and $\bm{\lambda}$.}
    \label{fig:fig6}
\end{figure}

The comparison of the SDC variants \texttt{SDC-C} and \texttt{SI-SDC} with \texttt{LU}, \texttt{MIN-SR-NS}, and \texttt{MIN-SR-S} preconditioning, respectively, the Radau IIA methods, and the \texttt{DOPRI5} method is shown in \prettyref{fig:fig5}. The parallel SDC variants indicate significantly reduced runtime compared to serial SDC methods. Based on the performance results, \texttt{SDC-C} demonstrates a minor runtime advantage over \texttt{SI-SDC}, suggesting a more cost-efficient formulation. Compared to the tested RK methods, the serial SDC approaches show inferior performance for this particular problem. In contrast, the parallel SDC schemes achieve higher performance than \texttt{RadauIIA5} and \texttt{DOPRI5} in terms of runtime. \texttt{RadauIIA7} remains computationally efficient in this setting, as it computes the solution within each time step by solving a single linear system of dimension $8$. In contrast, SDC methods require the solution of multiple linear systems of dimension $2$ at every collocation node in each iteration, which leads to a significantly higher overall runtime. However, SDC variants offer the advantage of maintaining a higher accuracy even for relatively large time steps, which makes them clearly more favorable than Runge–Kutta methods with respect to accuracy. The numerical solution computed by an SDC method achieves higher accuracy than the solution computed by RK methods with the same time step size.

\begin{figure}[!t]
    \centering
    \includegraphics[width=\textwidth]{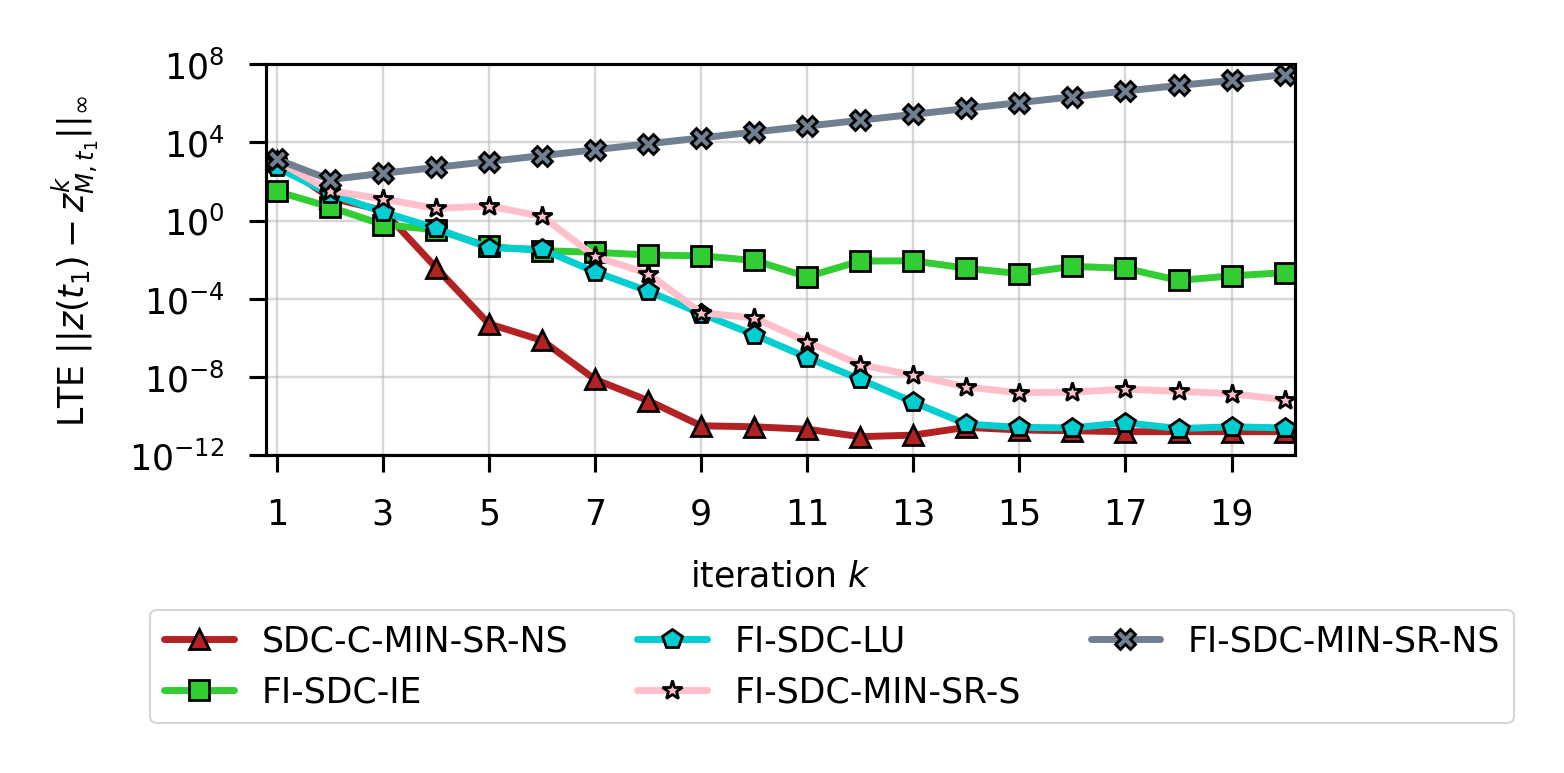}
    \caption{Error in algebraic variables $\bm{z} = (\bm{w}, \bm{\lambda})$ across iterations of \texttt{SDC-C-MIN-SR-NS} and \texttt{FI-SDC} variants using \texttt{IE}, \texttt{LU}, \texttt{MIN-SR-S}, and \texttt{MIN-SR-NS} preconditioning for Andrews' squeezer \prettyref{eq:andrews_index_one} in the first time step of size $\Delta t = 0.001$. The schemes are based on $M = 6$ nodes.}
    \label{fig:fig7}
\end{figure}

\subsection{Andrews' squeezer} \label{sec:andrews_squeezer}
Andrews' squeezing mechanism describes the motions of seven rigid bodies \cite{Andrews1986}. The index-$1$ formulation of the problem is given by
\begin{equation}
    \begin{split}
        \bm{q}' (t) &= \bm{v} (t), \\
        \bm{v}' (t) &= \bm{w} (t), \\
        \bm{0} &= \bm{M}(\bm{q}(t)) \bm{w} (t) - \bm{f}(\bm{q}(t), \bm{v} (t)) + \bm{G}^\top (\bm{q} (t)) \bm{\lambda} (t), \\
        \bm{0} &= \bm{g}_{\bm{qq}} (\bm{q}(t))(\bm{v}(t), \bm{v}(t)) + \bm{G}(\bm{q} (t)) \bm{w} (t)
    \end{split} \label{eq:andrews_index_one}
\end{equation}
with differential variables $\bm{q}(t), \bm{v}(t) \in \mathbb{R}^7$, and algebraic variables $\bm{w}(t) \in \mathbb{R}^7$, $\bm{\lambda}(t) \in \mathbb{R}^6$ for $t \in [0, T]$ with $T = 0.03$. In the original problem formulation of index three, the constraint
\begin{equation}
    \bm{0} = \bm{g}_{\bm{qq}}(\bm{q}(t))(\bm{v}(t), \bm{v}(t)) + \bm{G}(\bm{q} (t)) \bm{w} (t)
\end{equation}
is replaced by
\begin{equation}
    \bm{0} = \bm{g} (\bm{q}(t)) \label{eq:index_three_constraint}
\end{equation}
and the DAE of index one \prettyref{eq:andrews_index_one} is then obtained by differentiating \prettyref{eq:index_three_constraint} twice \cite[Chap.~VII.7]{Hairer_stiff2010}. The setup with the explicit functions and the matrices can be found in the reference just cited.

The implicit system at each collocation node is solved by Newton's method with tolerance $tol_{\mathrm{newton}} = 10^{-14}$. We choose the increment tolerance $\etol = 10^{-9}$.

\prettyref{fig:fig6} shows the order of accuracy after each \texttt{SDC-C-MIN-SR-NS} iteration. In $\bm{q}$, $\bm{w}$, and $\bm{\lambda}$ the provisional solution is second-order accurate. After $k = 1$ and $k = 2$ iterations, the solution gains even two orders. Thus, the maximum order of the SDC scheme (depending on the underlying quadrature rule) is reached after three iterations. This is different to the order observed for $\bm{v}$: The provisional solution is first-order accurate as expected, and the order of the numerical solution jumps by two after $k = 1$ and $k = 3$ iterations. Therefore, the numerical solution achieves the maximum order of accuracy after four iterations. As for the linear problem, the occurrence of order jumps cannot be explained, and no prediction is yet possible for them, but costs are saved to achieve the desired accuracy. 

In \prettyref{fig:fig7}, the global error in the algebraic variables versus the iterations is shown for various SDC schemes in the first time step of size $\Delta t = 0.001$. In \texttt{FI-SDC}, the numerical integration of algebraic variables decelerates the convergence rate, thereby reducing the overall efficiency of the method. In particular, the $\bm{Q}_\Delta^{\texttt{IE}}$ and $\bm{Q}_\Delta^{\texttt{MIN-SR-S}}$ preconditioners lead to noticeably slower convergence. Note that the former matrix is well known to lead to poor convergence properties in the scheme. In contrast, the \texttt{FI-SDC-LU} method appears to effectively mitigate this slowdown. The \texttt{FI-SDC-MIN-SR-NS} method diverges,
which is consistent with what we have shown in \prettyref{rem:alg_stiff_limit}: The numerical integration of algebraic variables represents a stiff limit, and non-stiff treatment leads to a divergent scheme. The SDC schemes based on $\bm{Q}_\Delta^{\texttt{Picard}}$ are not shown here, because the explicit \texttt{FI-SDC-Picard} method leads to a linear system with singular coefficient matrix.

\begin{figure}[!t]
    \centering
    \includegraphics[width=\textwidth]{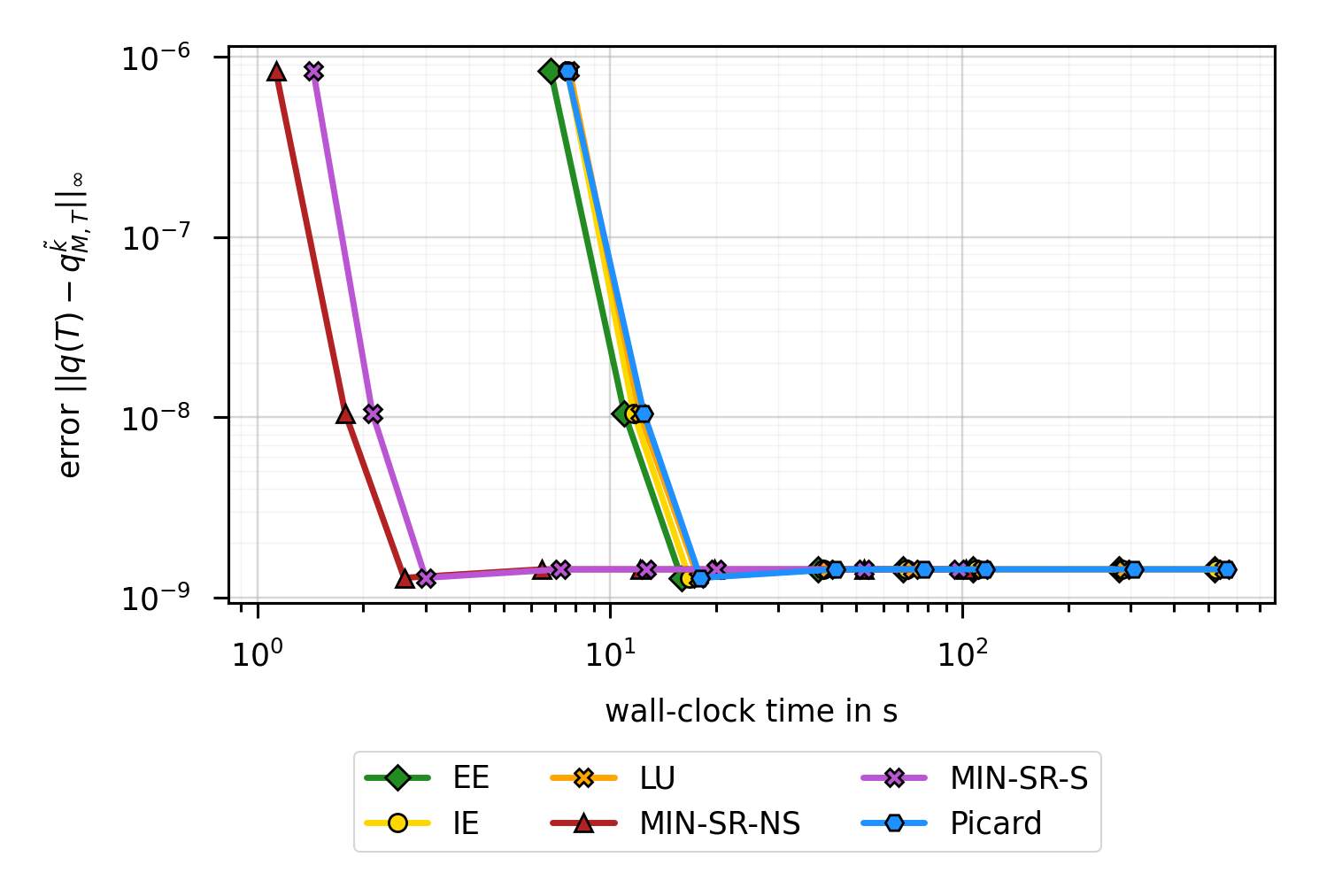}
    \caption{Wall-clock time against maximum absolute error in $\bm{q}$ at end time $T = 0.03$ of different \texttt{SDC-C} methods for Andrews' squeezer \prettyref{eq:andrews_index_one}. The schemes are based on $M = 6$ Radau IIA nodes using several choices of $\bm{Q}_\Delta$. All methods used the same time step sizes.}
    \label{fig:fig8}
\end{figure}

The wall-clock time versus the error in $\bm{q}$ at end time $T$ for different \texttt{SDC-C} methods based on $M = 6$ nodes is shown in \prettyref{fig:fig8}. Both coefficients, \texttt{MIN-SR-NS} and \texttt{MIN-SR-S}, accelerate the \texttt{SDC-C} method with a similar effort where the \texttt{MIN-SR-NS} coefficients are more efficient since it is designed for non-stiff problems and the differential equations are non-stiff. Parallel computation saves computational time and gives an advantage over all serial methods. The explicit \texttt{SDC-C-EE} method outperforms all other serial schemes that can be expected due to the non-stiffness. Obviously, the numerical solution does not gain in accuracy for $\Delta t \le 3 \cdot 10^{-4}$. Therefore, parallel methods can use a larger time step size than the serial methods to compute a solution with high accuracy.

\begin{figure}[!t]
    \centering
    \includegraphics[width=\textwidth]{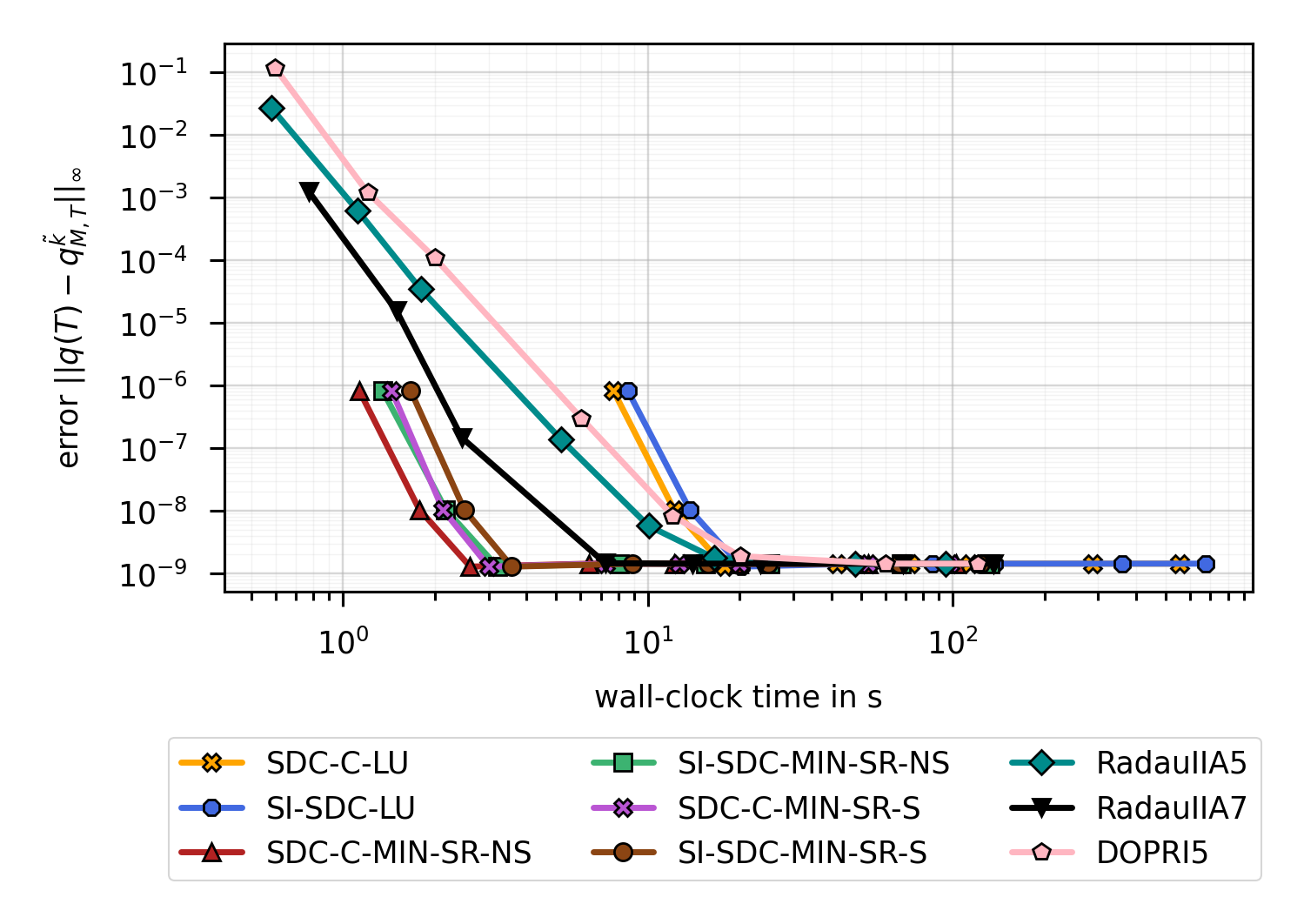}
    \caption{Wall-clock time versus maximum absolute error in $\bm{q}$ at end time $T = 0.03$ of different SDC variants and RK methods for Andrews' squeezer \prettyref{eq:andrews_index_one}. The SDC schemes are based on $M = 6$ using \texttt{LU}, \texttt{MIN-SR-NS}, and \texttt{MIN-SR-S} preconditioning. All methods used the same time step sizes.}
    \label{fig:fig9}
\end{figure}

In \prettyref{fig:fig9}, the comparison of \texttt{SDC-C} and \texttt{SI-SDC} with \texttt{LU}, \texttt{MIN-SR-NS}, and \texttt{MIN-SR-S} preconditioning, the Radau IIA methods, and the \texttt{DOPRI5} method is shown. Again, the parallel SDC methods show a clear advantage in computational runtime over the serial schemes. The "LU-trick" is designed for stiff problems and thus works worse for non-stiff problems resulting in longer runtimes. Compared with the RK methods, the parallel schemes clearly outperform them because the Radau IIA schemes are less efficient for larger dense systems to be solved in each time step. The \texttt{DOPRI5} method also needs more time to compute the systems at each stage via forward substitution, and the benefit of parallelized schemes becomes visible. The \texttt{SDC-C-MIN-SR-NS} method computes a numerical solution with an accuracy of $1.4 \cdot 10^{-9}$ almost ten times faster than \texttt{DOPRI5}, $7.8$ times faster than \texttt{RadauIIA5}, and $3.5$ times faster than \texttt{RadauIIA7}. The \texttt{SI-SDC} method with \texttt{MIN-SR-NS} preconditioning is competitive, as it computes a numerical solution for the same accuracy faster by a factor of $7.9$ than \texttt{DOPRI5}, $6.1$ times faster than \texttt{RadauIIA5}, and $2.7$ times faster than \texttt{RadauIIA7}.

\begin{figure}[!t]
    \centering
    \includegraphics[width=\textwidth]{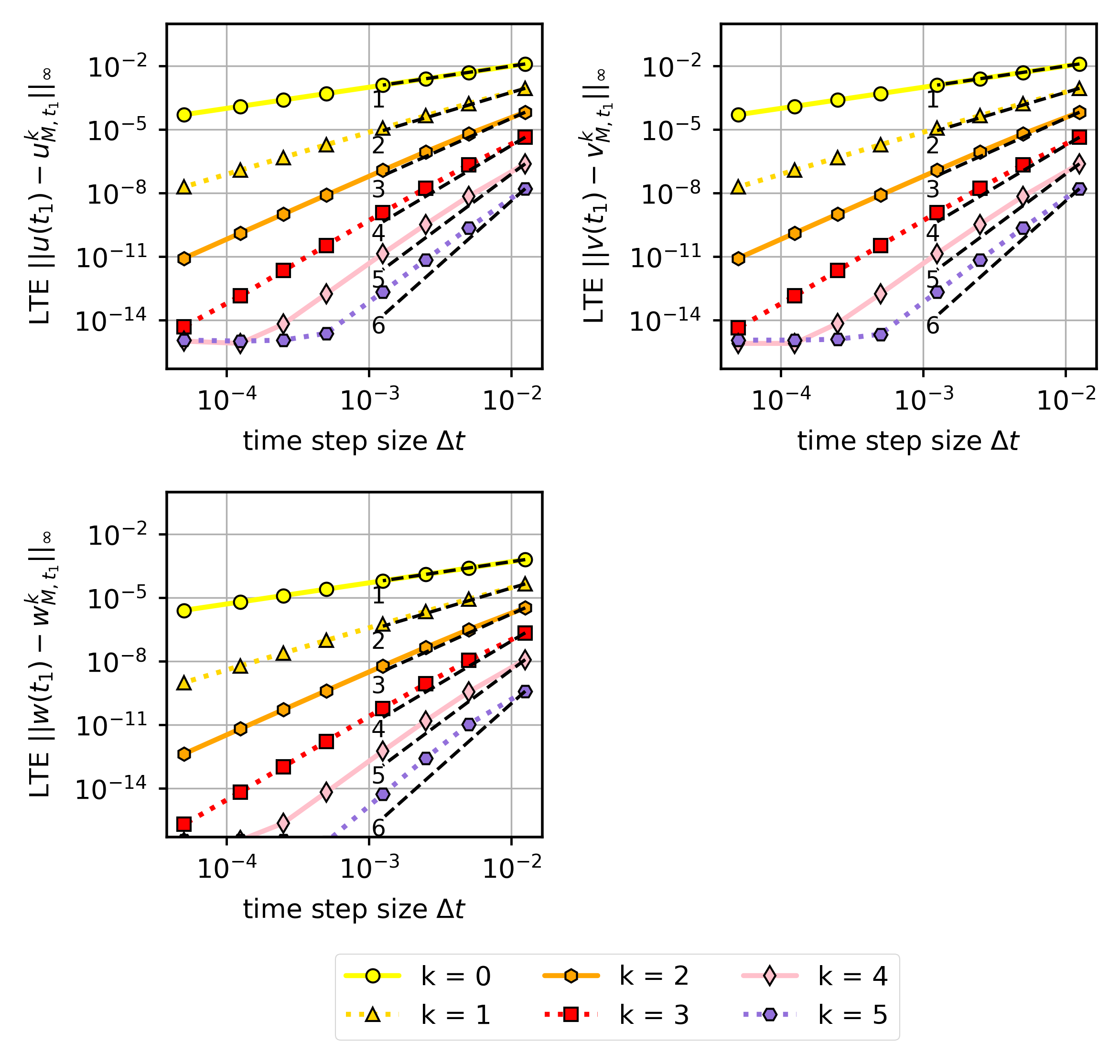}
    \caption{Local truncation error in $\bm{u}$, $\bm{v}$, and $\bm{w}$ against time step sizes $\Delta t$ for each iteration $k = 0,\dots,2M - 1$. Orders of accuracy are shown for \texttt{SDC-C-IE} based on $M = 3$ for the reaction-diffusion problem \prettyref{eq:reacdiff_problem}. Black dashed lines indicate the reference order. Top row: Errors in $\bm{u}$ and $\bm{v}$. Bottom row: Error in $\bm{w}$.}
    \label{fig:fig10}
\end{figure}

\subsection{Reaction-diffusion problem}
We consider the stiff reaction-diffusion problem formulated as a partial DAE (PDAE)
\begin{equation}
    \begin{split} \label{eq:reacdiff_problem}
        \frac{\partial \bm{u}}{\partial t} (\bm{x}, t) &= \frac{\partial^2 \bm{u}}{\partial \bm{x}^2} (\bm{x}, t) + \bm{u} (\bm{x}, t) \frac{\partial \bm{w}}{\partial \bm{x}} (\bm{x}, t) + \bm{\tilde{f}}(\bm{x}, t), \\
        \frac{\partial \bm{v}}{\partial t} (\bm{x}, t) &= \frac{\partial^2 \bm{v}}{\partial \bm{x}^2} (\bm{x}, t) - \bm{v} (\bm{x}, t) \frac{\partial \bm{w}}{\partial \bm{x}} (\bm{x}, t) + \bm{\tilde{g}}(\bm{x}, t), \\
        \bm{0} &= -\bm{u} (\bm{x}, t) - \bm{v} (\bm{x}, t) - \frac{\partial^2 \bm{w}}{\partial \bm{x}^2} (\bm{x}, t)
    \end{split}
\end{equation}
with differential variables $\bm{u}(\bm{x}, t), \bm{v}(\bm{x}, t) \in \mathbb{R}^{n_x}$ and algebraic variable $\bm{w}(\bm{x}, t) \in \mathbb{R}^{n_x}$ being concentrations and source terms $\bm{\tilde{f}}(\bm{x}, t), \allowbreak \bm{\tilde{g}}(\bm{x}, t) \in \mathbb{R}^{n_x}$ in the time interval $[0, 0.25]$. We consider the problem on the spatial grid $x_i = i \Delta x$, $\Delta x = \frac{1}{n_x}$ in $[0, 1]$ for $i = 0,..,n_x - 1$ with $n_x = 256$ degrees of freedom. The source terms are constructed such that the exact solutions of the problem are given by
\begin{equation*}
    \begin{split}
      \bm{u}(\bm{x}, t) &= A \sin(2\pi \bm{x}) \exp(t), \qquad \bm{v}(\bm{x}, t) = B \sin(2\pi \bm{x}) \exp(t), \\
      &\qquad\qquad\bm{w}(\bm{x}, t) = \tfrac{A + B}{4 \pi^2}\,\sin(2\pi \bm{x}) \exp(t)
    \end{split}
\end{equation*}
with $A = B = -1$, i.e.,
\begin{equation*}
    \begin{split}
        \bm{\tilde{f}} (\bm{x}, t) &= \frac{\partial \bm{u}}{\partial t} (\bm{x}, t) - \frac{\partial^2 \bm{u}}{\partial \bm{x}^2} (\bm{x}, t) - \bm{u} (\bm{x}, t) \frac{\partial \bm{w}}{\partial \bm{x}} (\bm{x}, t), \\
        \bm{\tilde{g}} (\bm{x}, t) &= \frac{\partial \bm{v}}{\partial t} (\bm{x}, t) - \frac{\partial^2 \bm{v}}{\partial \bm{x}^2} (\bm{x}, t) + \bm{v} (\bm{x}, t) \frac{\partial \bm{w}}{\partial \bm{x}} (\bm{x}, t).
    \end{split}
\end{equation*}
Periodic boundary conditions
\begin{equation*}
    \bm{u}(0, t) = \bm{u}(1, t), \quad \bm{v}(0, t) = \bm{v}(1, t), \quad \bm{w}(0, t) = \bm{w}(1, t)
\end{equation*}
are chosen. Initial conditions are obtained by evaluating the exact solutions at initial time $t_0 = 0$, i.e., $\bm{u}_0(\bm{x}) = \bm{u}(\bm{x}, 0), \bm{v}_0(\bm{x}) = \bm{v}(\bm{x}, 0)$ and $\bm{w}_0(\bm{x}) = \bm{w}(\bm{x}, 0)$. For a comprehensive analysis to local and global solutions, the reader is referred to \cite{Benabdallah2025}. The computations have been performed in the spectral space using the discrete Fourier transformation and afterwards shifted back to the real space, ensuring high spatial accuracy of the numerical solutions. Newton's method is used to solve the implicit system in spectral space at each node. Here, no space parallelism is used. The tolerance is coupled to the time step size by
\begin{equation*}
    \tolnewton = \frac{\tolref \cdot \Delta t}{\dtref} 
\end{equation*}
with $\tolref = 1.3 \cdot 10^{-12}$ and $\dtref = 2.6 \cdot 10^{-3}$ to balance accuracy and efficiency. For reference time step sizes $\Delta t_{\mathrm{ref},1}, \Delta t_{\mathrm{ref},2} > 0$ with $\Delta t_{\mathrm{ref},1} < \Delta t_{\mathrm{ref},2}$, the corresponding Newton tolerances satisfy $tol_{\mathrm{newton},1} > tol_{\mathrm{newton},2}$ for fixed $\tolref$. Therefore, choosing $\Delta t > \dtref$, a stricter tolerance is imposed to match the higher temporal accuracy, while for $\Delta t < \dtref$ a relaxed tolerance can reduce the computational cost of nonlinear solves. In the former case, the Newton iteration typically terminates after the maximum number of iterations without reaching the prescribed tolerance, but still benefits the overall solution accuracy. We assume reference tolerances $tol_{\mathrm{ref},1}, tol_{\mathrm{ref},2} > 0$ with $tol_{\mathrm{ref},1} < tol_{\mathrm{ref},2}$, we obtain Newton tolerances that satisfy $tol_{\mathrm{newton},1} < tol_{\mathrm{newton},2}$ for fixed $\dtref$. Accordingly, for fixed reference step size $\dtref$ and arbitrary step size $\Delta t$, choosing a relaxed reference tolerance $\tolref > 1.3 \cdot 10^{-12}$ deteriorates the precision by reducing computational costs, while higher accuracy and higher costs are obtained for $\tolref < 1.3 \cdot 10^{-12}$. This strategy with specific chosen $\tolref$ and $\dtref$ above provides an effective compromise between runtime and accuracy. The increment tolerance is set to $\etol = 10^{-12}$.

\begin{figure}[!t]
    \centering
    \includegraphics[width=\textwidth]{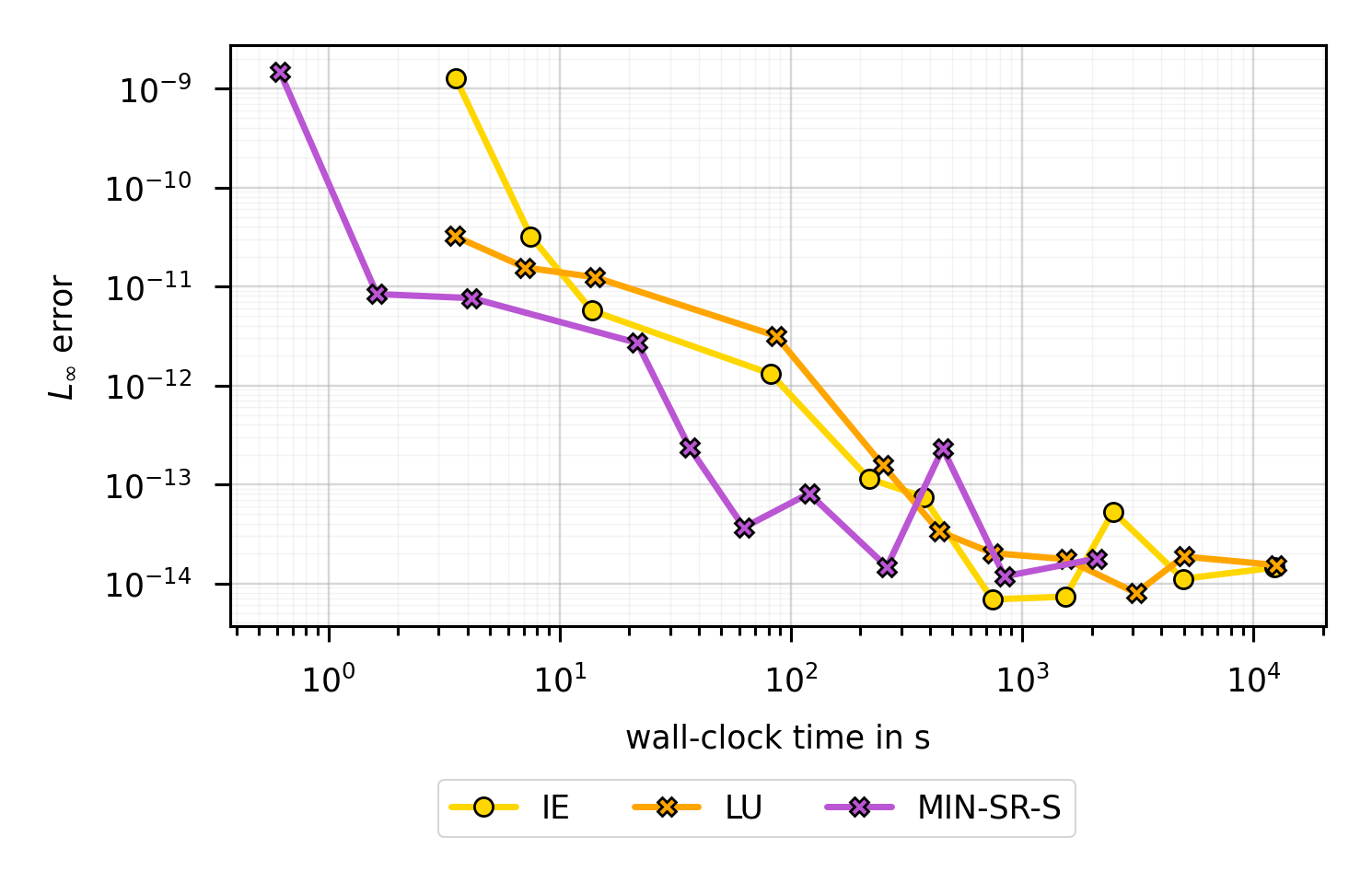}
    \caption{Wall-clock time against $L_\infty$ error of different \texttt{SDC-C} methods for the reaction-diffusion problem \prettyref{eq:reacdiff_problem}. The schemes are based on $M = 6$ Radau IIA nodes using different choices of $\bm{Q}_\Delta$. All methods used the same time step sizes.}
    \label{fig:fig11}
\end{figure}

In \prettyref{fig:fig10}, the order of accuracy for $\bm{u}, \bm{v}$, and $\bm{w}$ after each \texttt{SDC-C-IE} iteration is shown. After $k = 3$ iterations, the expected orders are barely achieved. Although the \texttt{IE} preconditioner is suitable for stiff problems, it is also known for its slow convergence, but this is observed for \texttt{LU} and \texttt{MIN-SR-S} preconditioning as well. In our numerical experiments (not shown here), the observed order in $\bm{w}$ for \texttt{SI-SDC} and \texttt{FI-SDC} after iteration $k$ does not fully match the prediction of \prettyref{thm:lte_sdc_c}. Especially, numerical integration inside the algebraic constraints leads to slower convergence of the algebraic constraints, implying that the convergence rate is reduced at least in $\bm{w}$, see \prettyref{rem:convergence_alg_const}. A comprehensive theoretical analysis could provide further insight into this observation, but is left for future work.

\prettyref{fig:fig11} shows the wall-clock time versus the $L_\infty$ error of \texttt{SDC-C} using the stiff choices of matrices $\bm{Q}_\Delta$. The \texttt{SDC-C} method benefits from the \texttt{MIN-SR-S} coefficients by saving computational time. It computes the solution quite faster than the serial schemes \texttt{SDC-C-IE} and \texttt{SDC-C-LU}. Although \texttt{SDC-C-IE} and \texttt{SDC-C-LU} show similar runtimes, \texttt{SDC-C-IE} yields a higher overall accuracy for moderate time step sizes. Obviously, the parallel \texttt{MIN-SR-S} method can compute a numerical solution using a larger step size than the serial \texttt{IE} and \texttt{LU} schenmes. The non-stiff choices $\bm{Q}_\Delta^{\texttt{EE}}$, $\bm{Q}_\Delta^{\texttt{Picard}}$, and $\bm{Q}_\Delta^{\texttt{MIN-SR-NS}}$ result in diverging methods due to the stiffness of the problem, where the latter matrix yields an implicit scheme.

\begin{figure}[!t]
    \centering
    \includegraphics[width=\textwidth]{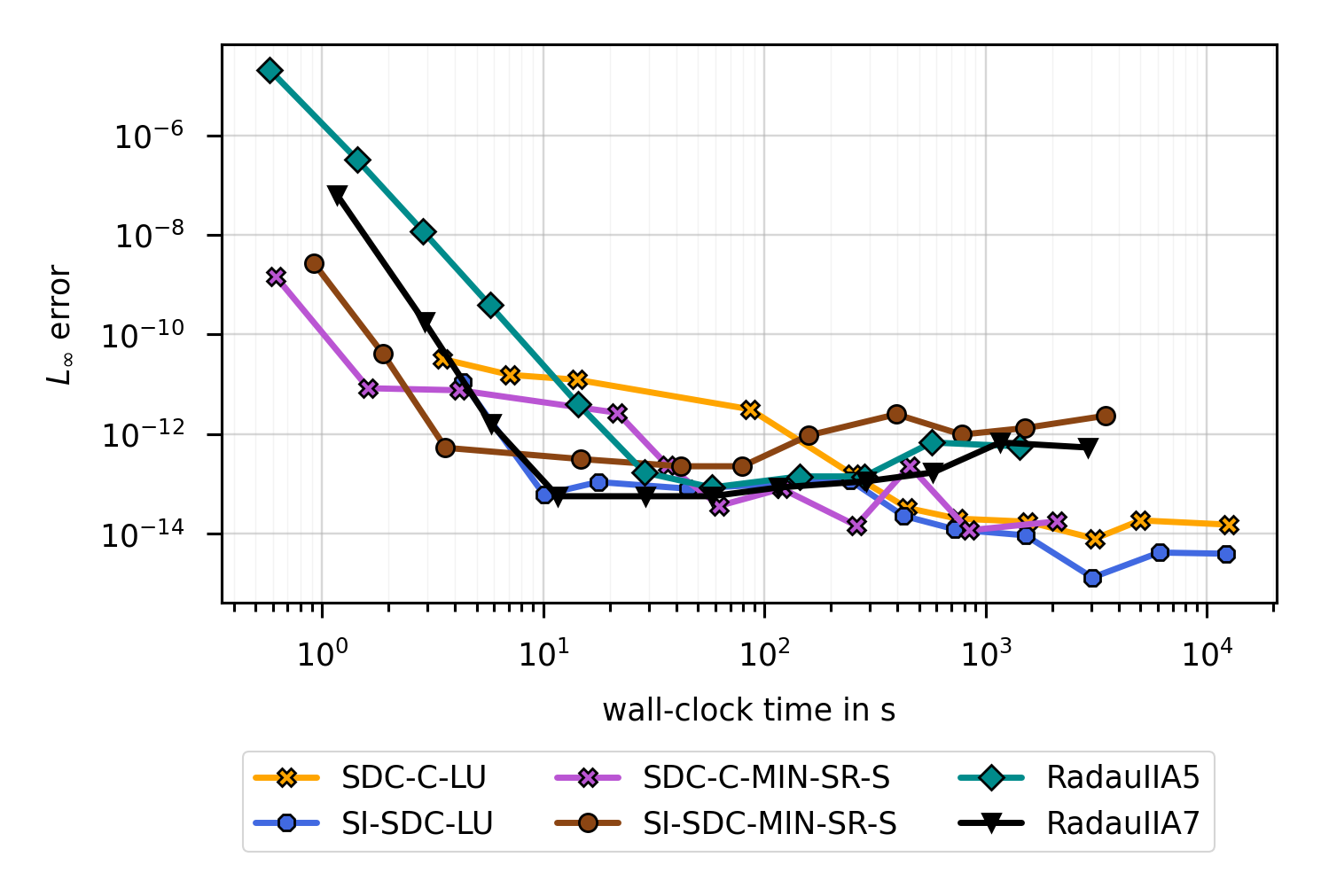}
    \caption{Wall-clock time versus $L_\infty$ error of different SDC variants and Radau methods for the reaction-diffusion problem \prettyref{eq:reacdiff_problem}. The SDC schemes are based on $M = 6$ collocation nodes using \texttt{LU} and \texttt{MIN-SR-S} preconditioning. All methods used the same time step sizes.}
    \label{fig:fig12}
\end{figure}

In \prettyref{fig:fig12}, the wall-clock time versus the $L_\infty$ error of different SDC variants and the Radau IIA methods are shown. For larger $\Delta t$ the \texttt{MIN-SR-S} schemes benefit from the coupled Newton tolerance to the time step size and outperform the Radau IIA method of order $7$ in runtime and accuracy. While the fifth-order Radau IIA method computes as fast as \texttt{SDC-C-MIN-SR-S} and faster than \texttt{SI-SDC-MIN-SR-S}, both SDC variants compute solutions of higher accuracy. The serial \texttt{LU} variants are less favorable here due to their high computational costs. The \texttt{SI-SDC-LU} scheme is much more accurate than the \texttt{SDC-C-LU} scheme although they have similar runtimes. The accuracy can be improved by adjusting the convergence parameters, but it also results in higher computational costs.

\begin{figure}[!t]
    \centering
    \includegraphics[width=\textwidth]{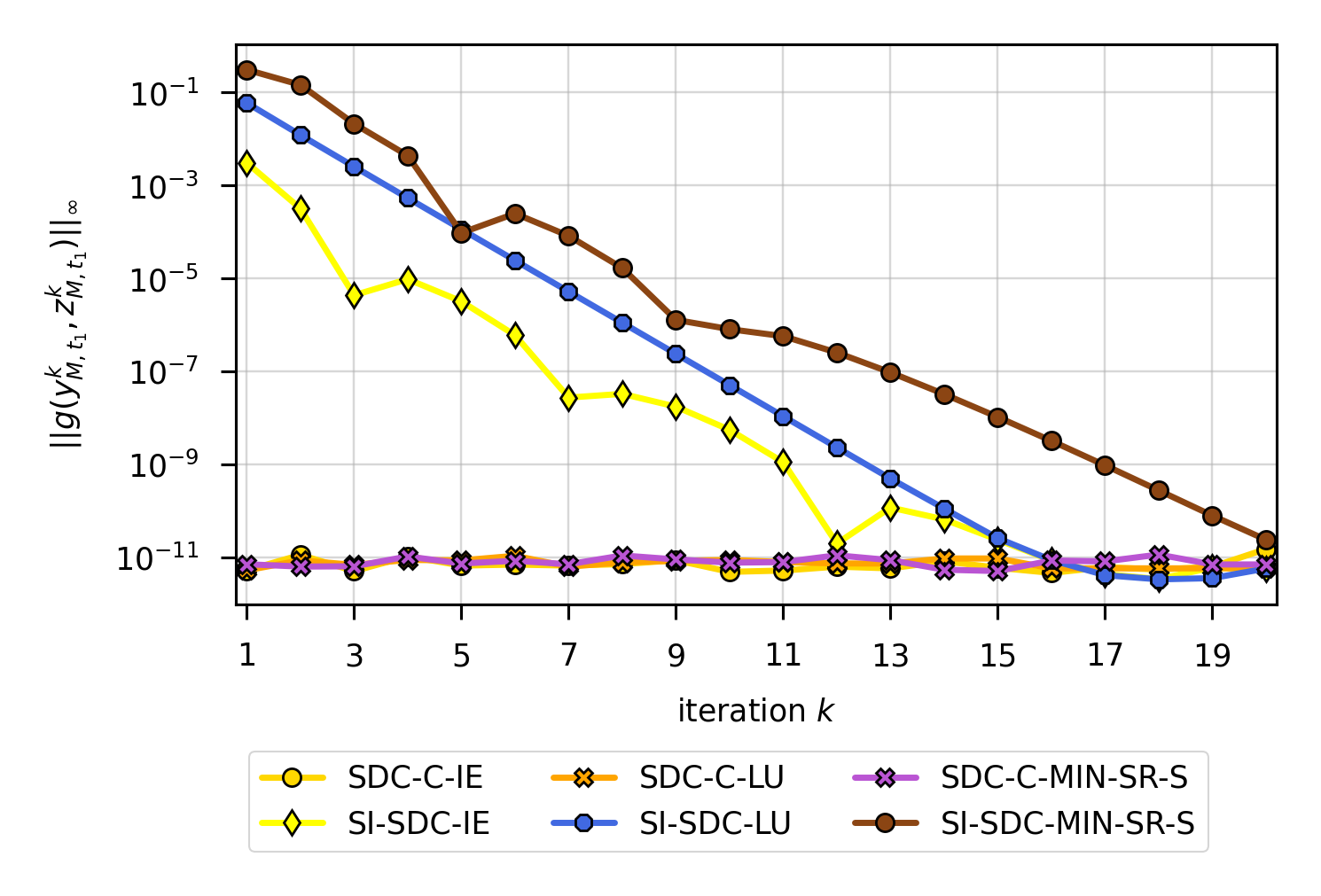}
    \caption{Maximum absolute value of algebraic constraints $\bm{g}$ across iterations of \texttt{SDC-C} and \texttt{SI-SDC} methods for the reaction-diffusion problem \prettyref{eq:reacdiff_problem} in the first time step of size $\Delta t = 0.125$. The schemes are based on $M = 6$ nodes using different $\bm{Q}_\Delta$. Here, we have $\bm{y} = (\bm{u}, \bm{v})$ and $\bm{z} = \bm{w}$.}
    \label{fig:fig13}
\end{figure}

For different \texttt{SDC-C} and \texttt{SI-SDC} schemes, the maximum absolute value of the algebraic constraints after iteration $k$ is shown in \prettyref{fig:fig13}. Although the \texttt{SI-SDC} scheme is suitable for semi-explicit DAEs, it cannot be expected that the algebraic equations are satisfied in each iteration but only in case of convergence, see \prettyref{rem:convergence_alg_const}. In contrast, the numerical solution computed by the \texttt{SDC-C} scheme always satisfies the algebraic constraints due to the construction of the scheme.

\section{Conclusion and outlook} \label{sec:conclusion}
The solution of DAEs requires a tailored solver to efficiently compute a numerical solution with high accuracy. The SDC method is a high-order scheme that iteratively computes a solution by correcting the provisional approximation after each iteration. Under certain assumptions, the numerical solution gains one order per iteration \cite{Shu07}. In \cite{Huang2007}, a SDC method for general IDEs and a semi-integrating variant for semi-explicit DAEs is proposed. The traditional idea of SDC is carried over to the semi-explicit DAE case to derive the \texttt{SDC-C} scheme. It only integrates the differential equations with spectral quadrature and retains the algebraic constraints as an implicit condition in the system. The efficiency of the scheme is studied for various choices of $\bm{Q}_\Delta$ matrices. Especially, we study the scheme for diagonal matrices $\bm{Q}_\Delta$ that allows the scheme to be parallelized across the method.

We have theoretically shown that each \texttt{SDC-C} iteration improves the numerical solution in differential and algebraic variables of index-one problems by one where we modified the proof given by Y. Xia et al. for semi-explicit DAEs. In numerical experiments, this has been confirmed in three test scenarios: The linear test DAE, the nonlinear problem describing Andrews' squeezing mechanism, and the reaction-diffusion problem as a PDAE. The efficiency of the different schemes was evaluated by measuring in runtime needed to compute a solution over a certain time domain. The proposed scheme was compared with other SDC methods and existing RK methods for DAEs that included a half-explicit RK scheme and Radau IIA methods. We demonstrated that the \texttt{SDC-C} scheme is competitive with all the different schemes. The proposed \texttt{SDC-C} method consistently has achieved a higher accuracy than the Radau IIA solvers for identical time step sizes. In addition, for several test cases, the solution computed by \texttt{SDC-C} has achieved high accuracy at a comparable or even reduced computational cost, resulting in shorter runtimes than those of the Radau IIA methods. This highlights the potential of \texttt{SDC-C} as an efficient and accurate alternative to classical implicit Runge–Kutta solvers. The presented scheme has also been shown to be more efficient than other SDC variants, highlighting its favorable balance between accuracy and efficiency.

Currently, the parallel properties of the \texttt{SDC-C} scheme are under investigation.

In this work, the analysis of the \texttt{SDC-C} method is restricted to index-one problems. In order to find out how the method performs for higher index problems, future work will focus on index-two problems. Recent works deal with the application of SDC for the incompressible Navier-Stokes equations using projection schemes \cite{MinionSaye2018}, \cite{Stiller2020}. The problem is of \textit{Hessenberg form} of index two \cite{Ascher1998} given by
\begin{equation}
    \frac{\partial \bm{v}}{\partial t}(\bm{x}, t) = \bm{f}(\bm{v} (\bm{x}, t), \bm{p} (\bm{x}, t), t), \qquad \qquad\bm{0} = \bm{g} (\bm{v}(\bm{x}, t), t),
\end{equation}
with velocity $\bm{v} (\bm{x}, t) \in \mathbb{R}^{n_d}$ as differential variable and pressure $\bm{p} (\bm{x}, t) \in \mathbb{R}^{n_a}$ as algebraic variable. The right-hand side of the momentum equation is denoted by $\bm{f}(\bm{v} (\bm{x}, t), \bm{p} (\bm{x}, t), t) \in \mathbb{R}^{n_d}$, and $\bm{g} (\bm{v}(\bm{x}, t), t) \in \mathbb{R}^{n_a}$ denotes the right-hand side of the continuity equation. One of the methods used in the cited works is algorithmically similar to the \texttt{SDC-C} method. In pressure correction, the momentum equation is solved with an approximation of the pressure and the continuity equation is preserved using a projection step to correct the pressure \cite{Stiller2020}. The proof can be extended to index-two problems by making stronger assumptions to Jacobian matrices of the functions, see the note at the end of \prettyref{sec:sdc_dae_lte}.
\section*{Data Availability}
The data supporting the findings of this study are generated by numerical simulations. 
The code used to reproduce all results is publicly available at
\url{https://github.com/lisawim/pySDC/tree/sdc_dae_analysis_paper}.


%

\section*{Funding}
The computations were carried out on the PLEIADES cluster at the University of Wuppertal, which was supported by the Deutsche Forschungsgemeinschaft (DFG, grant No. INST 218/78-1 FUGG) and the Bundesministerium für Bildung und Forschung (BMBF).

\section*{Declarations}

\section*{Conflict of Interest}
The author declares that there is no conflict of interest.

\section*{Competing Interests}
The author declares no competing interests.

\section*{Ethics Approval}
Not applicable.

\section*{Consent to Participate}
Not applicable.

\section*{Consent for Publication}
Not applicable.

\bibliographystyle{spmpsci}      
\bibliography{references}   


\end{document}